\documentclass[12pt]{article}		%** USE THIS FOR THE ELECTRONIC JOURNAL OF COMB.
\usepackage{e-jc}			%** USE THIS FOR THE ELECTRONIC JOURNAL OF COMB.

\usepackage{stmaryrd}		% Lets you use \mapsfrom command
\usepackage{amsfonts}
\usepackage{amssymb}
\usepackage{amsmath}
\usepackage{amsthm}
\usepackage[usenames]{color}
\usepackage{amscd}
\usepackage{verbatim}

\usepackage{subfigure}	%Use this to have figures placed side by side!!
\usepackage{cancel}	%This allows to strikethrough command \sout{whatever}
\usepackage[small,bf]{caption}
\usepackage{graphicx}
\usepackage{enumerate}
\usepackage{float}		%Allows you to use \begin{figure}[H] command
\usepackage[colorlinks=true, citecolor=blue, linktocpage]{hyperref}%This will make hyperlinks to labelled parts!!!

\newtheorem{theorem}{Theorem}[section] %[section] here ensures the start of each section has sets the theorem counter back to 1. 
\newtheorem{lemma}[theorem]{Lemma}
\newtheorem{corollary}[theorem]{Corollary}
\newtheorem{conjecture}[theorem]{Conjecture}

\newtheorem{proposition}[theorem]{Proposition}

\newtheorem*{nonumthm}{Theorem}

\theoremstyle{definition}
\newtheorem{example}[theorem]{Example}

\theoremstyle{definition}
\newtheorem{definition}[theorem]{Definition}

\theoremstyle{plain}

\theoremstyle{definition}
\newtheorem{remark}[theorem]{Remark}

\theoremstyle{definition}
\newtheorem*{property}{Subfilling Property}

\DeclareMathOperator{\GL}{GL}		%This is the General Linear Group
\DeclareMathOperator{\DP}{DP}		%This is the set of Dimension Pairs
\DeclareMathOperator{\Mat}{Mat}		%This is for the group of n by n matrices
\DeclareMathOperator{\JT}{(\textit{h},\mu)}

\newcommand{\LTJ}{\left\langle LT(J_h)\right\rangle}
\newcommand{\X}{\textbf{x}^\alpha}

\newcommand{\A}{\mathcal{A}(\mu)}
\newcommand{\B}{\mathcal{B}(\mu)}
\newcommand{\D}{\DP^T}

\newcommand{\Ah}{\mathcal{A}_h(\mu)}
\newcommand{\Bh}{\mathcal{B}_h(\mu)}

\newcommand{\FiveTab}{\setlength{\unitlength}{.15in}\begin{picture}(5,1)(0,0)
\linethickness{.25pt}
\multiput(0,0)(0,1){2}{\line(1,0){5}}
\put(0,0){\line(0,1){1}}
\put(5,0){\line(0,1){1}}
\end{picture}}

\newcommand{\TwoOneTab}[4][1]{\scalebox{#1}{\begin{tabular}{|c|c|} \hline
  \begin{LARGE}#2\end{LARGE} & \begin{LARGE}#3\end{LARGE} \\ \hline
  \begin{LARGE}#4\end{LARGE} \\ \cline{1-1}
\end{tabular}}}

\numberwithin{figure}{section}
\numberwithin{theorem}{subsection}

\input xy	%ZAJJ DAUGHERTY showed me xy-pic at MSRI
\xyoption{all}

\title{\vspace*{-50pt}\textbf{A Hessenberg generalization of the Garsia-Procesi basis for the cohomology ring of Springer varieties}}
\author{Aba Mbirika\\
\small Department of Mathematics\\[-0.8ex]
\small Bowdoin College\\[-0.8ex]
\small Brunswick, Maine, USA\\[-0.8ex]
\small \texttt{ambirika@bowdoin.edu}\\[-0.8ex]
\small www.bowdoin.edu/$\sim$ambirika}

\makeatletter
\newcommand\ackname{Acknowledgments}
\if@titlepage
  \newenvironment{acknowledgments}{%
      \titlepage
      \null\vfil
      \@beginparpenalty\@lowpenalty
      \begin{center}%
        \bfseries \ackname
        \@endparpenalty\@M
      \end{center}}%
     {\par\vfil\null\endtitlepage}
\else
  \newenvironment{acknowledgments}{%
      \if@twocolumn
        \section*{\abstractname}%
      \else
        \small
        \begin{center}%
          {\bfseries \ackname\vspace{-.5em}\vspace{\z@}}%
        \end{center}%
        \quotation
      \fi}
      {\if@twocolumn\else\endquotation\fi}
\fi
\makeatother

\date{\dateline{Jan 7, 2010}{Oct 29, 2010}{Nov 11, 2010}}

\begin{document}

\maketitle
\vspace{-.25in}
\begin{abstract}
The Springer variety is the set of flags stabilized by a nilpotent operator.  In 1976, T.A. Springer observed that this variety's cohomology ring carries a symmetric group action, and he offered a deep geometric construction of this action.  Sixteen years later, Garsia and Procesi made Springer's work more transparent and accessible by presenting the cohomology ring as a graded quotient of a polynomial ring.  They combinatorially describe an explicit basis for this quotient.  The goal of this paper is to generalize their work.  Our main result deepens their analysis of Springer varieties and extends it to a family of varieties called Hessenberg varieties, a two-parameter generalization of Springer varieties.  Little is known about their cohomology.  For the class of regular nilpotent Hessenberg varieties, we conjecture a quotient presentation for the cohomology ring and exhibit an explicit basis.  Tantalizing new evidence supports our conjecture for a subclass of regular nilpotent varieties called Peterson varieties.
\end{abstract}

%\newpage
\tableofcontents

%\newpage
\section{Introduction}\label{sec:Intro}
The Springer variety $\mathfrak{S}_X$ is defined to be the set of flags stabilized by a nilpotent operator $X$.  Each nilpotent operator corresponds to a partition $\mu$ of $n$ via decomposition of $X$ into Jordan canonical blocks.  In 1976, Springer~\cite{Spr76} observed that the cohomology ring of $\mathfrak{S}_X$ carries a symmetric group action, and he gave a deep geometric construction of this action.  In the years that followed, De Concini and Procesi~\cite{dCP} made this action more accessible by presenting the cohomology ring as a graded quotient of a polynomial ring.  Garsia and Procesi~\cite{GP} later gave an explicit basis of monomials $\B$ for this quotient.  Moreover, they proved this quotient is indeed isomorphic to $H^*(\mathfrak{S}_X)$.

We explore the two-parameter generalization of the Springer variety called Hessenberg varieties $\mathfrak{H}(X,h)$, which were introduced by De Mari, Procesi, and Shayman~\cite{dMPS}.  These varieties are parametrized by a nilpotent operator $X$ and a nondecreasing map $h$ called a Hessenberg function.  The cohomology of Springer's variety is well-known~\cite{Spr76,Spr,dCP,GP,Tani}, but little is known about the cohomology of the family of Hessenberg varieties.  However in 2005, Tymoczko~\cite{Tym} offered a first glimpse by giving a paving by affines of these Hessenberg varieties.  This allowed her to give a combinatorial algorithm to compute its Betti numbers.  Using certain Young diagram fillings, which we call \textit{$\JT$-fillings} in this paper, she calculates the number of \textit{dimension pairs} for each $\JT$-filling.
\begin{nonumthm}[Tymoczko]
The dimension of $H^{2k}(\mathfrak{H}(X,h))$ is the number of $\JT$-fillings $T$ such that $T$ has $k$ dimension pairs.
\end{nonumthm}

A main result in this paper connects the dimension-counting objects, namely the $\JT$-fillings, for the graded parts of $H^*(\mathfrak{H}(X,h))$ to a set of monomials $\Ah$.  We describe a map $\Phi$ from these $\JT$-fillings onto the set $\Ah$ in Subsection~\ref{subsec:Phi_map}.  It turns out in the Springer setting that this map extends to a graded vector space isomorphism between two different presentations of cohomology, one geometric and the other algebraic.  Furthermore, the monomials $\Ah$ correspond exactly to the Garsia-Procesi basis (see Subsection~\ref{subsec:A_equals_B_monomials}) in this Springer setting.

For arbitrary non-Springer Hessenberg varieties $\mathfrak{H}(X,h)$, the map $\Phi$ takes $\JT$-fillings to a different set of monomials.  The natural question to ask is, ``Are the new corresponding monomials $\Ah$ meaningful in this setting?''.  For a certain subclass of Hessenberg varieties called \textit{regular nilpotent}, the answer is yes.  This is shown in Section~\ref{sec:RegNilpHess_Setting}.  We easily construct a special ideal $J_h$ (see Subsection~\ref{subsec:forthcomingwork}) with some interesting properties.  The quotient of a polynomial ring by this ideal has basis $\Bh$ which coincides exactly with the set of monomials $\Ah$.  Recent work of Harada and Tymoczko suggests that our quotient may be a presentation for $H^*(\mathfrak{H}(X,h))$ when $X$ is regular nilpotent.  Little is known about the cohomology of arbitrary Hessenberg varieties in general.  We hope to extend results to this setting in future work.  We illustrate this goal in Figure~\ref{fig:The_Goal}.
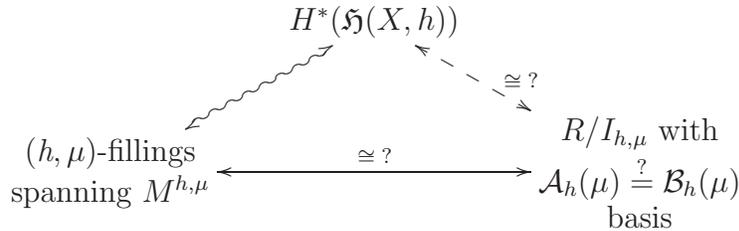
\begin{figure}[!ht]
$$
\xymatrix{
& H^*(\mathfrak{H}(X,h)) \ar@{<~>}[dl] \ar@{<-->}[dr]^{\cong\;?}\\
\txt{$\JT$-fillings\\spanning $M^{h,\mu}$}\ar@{<->}[rr]^{\cong\;?} & & \txt{$R/{I_{h,\mu}}$ with\\$\Ah\buildrel?\over=\Bh$\\basis}
}
$$
\vspace{-.15in}
\caption{\label{fig:The_Goal}Goal in the arbitrary Hessenberg setting.}
\end{figure}

\noindent The main results of this paper are the following:
\begin{itemize}
 \item In Section~\ref{sec:Springer_Setting}, we complete the three legs of the triangle in the Springer setting.  In this setting, the $\JT$-fillings are simply the row-strict tableaux.  They are the generating set for the vector space which we call $M^{\mu}$ (see Subsection~\ref{subsec:Remarks_on_Phi}).  The ideal $I_{h,\mu}$ in Figure~\ref{fig:The_Goal} is the famed Tanisaki ideal~\cite{Tani}, denoted $I_\mu$ in the literature.  It turns out that our set of monomials $\Ah$ coincides with the Garsia-Procesi basis $\B$ of monomials for the rational cohomology of the Springer varieties for $R:=\mathbb{Q}[x_1,\ldots,x_n]$.  Garsia and Procesi used a tree on Young diagrams to find $\B$.  We refine their construction and build a \textit{modified GP-tree for $\mu$} (see Definition~\ref{def:Modified_GP_Tree}).  This refinement helps us obtain more information from their tree, thus revealing our $\JT$-fillings in their construction of the basis.
 \item For each Hessenberg function $h$, we construct an ideal $J_h$ (see Subsection~\ref{subsec:forthcomingwork}) out of modified complete symmetric functions.  We identify a basis for the quotient $R/J_h$, where we take $R$ to be the ring $\mathbb{Z}[x_1,\ldots,x_n]$.
 \item To show that the bottom leg of the triangle holds in the regular nilpotent case, we construct what we call an \textit{$h$-tableau-tree} (see Definition~\ref{def:h_tab_tree}).  This tree plays the same role as its counterpart, the modified GP-tree, does in the Springer setting.  We find that the monomials $\Ah$ coincide with a natural basis $\Bh$ of monomials for $R/J_h$ (see Subsection~\ref{subsec:Ah_equals_Bh_RegNilpCase}).
 \item Recent results of Harada and Tymoczko~\cite{HT} give tantalizing evidence that the quotient $R/J_h$ may indeed be a presentation for $H^*(\mathfrak{H}(X,h))$ for a subclass of regular nilpotent Hessenberg varieties called Peterson varieties.  We conjecture $R/J_h$ is a presentation for the integral cohomology ring of the regular nilpotent Hessenberg varieties.
\end{itemize}

\newpage
\begin{acknowledgments}
The author thanks his advisor in this project, Julianna Tymoczko, for endless feedback at our many meetings.  Thanks also to Megumi Harada and Alex Woo for fruitful conversations.  He is also grateful to Fred Goodman for very helpful comments which significantly improved this manuscript.  Jonas Meyer and Erik Insko also gave useful input.  Lastly, I thank the anonymous referee for an exceptionally thorough reading of this manuscript and many helpful suggestions.
\end{acknowledgments}

\subsection{Brief history of the Springer setting}
Let $\mathfrak{N}(\mu)$ be the set of nilpotent elements in $\Mat_n(\mathbb{C})$ with Jordan blocks of weakly decreasing sizes $\mu_1\geq\mu_2\ldots\geq\mu_s>0$ so that $\sum_{i=1}^s \mu_i = n$.  The quest began 50 years ago to find the equations of the closure $\overline{\mathfrak{N}(\mu)}$ in $\Mat_n(\mathbb{C})$---that is, the generators of the ideal of polynomial functions on $\Mat_n(\mathbb{C})$ which vanish on $\mathfrak{N}(\mu)$.  When $\mu=(n)$, Kostant~\cite{Ko} showed in his fundamental 1963 paper that the ideal is given by the invariants of the conjugation action of $\GL_n(\mathbb{C})$ on $\Mat_n(\mathbb{C})$.  In 1981, De Concini and Procesi~\cite{dCP} proposed a set of generators for the ideals of the schematic intersections $\overline{\mathfrak{N}(\mu)} \cap T$ where $T$ is the set of diagonal matrices and $\mu$ is an arbitrary partition of $n$.  In 1982, Tanisaki~\cite{Tani} simplified their ideal; his simplification has since become known as the Tanisaki ideal $I_\mu$.  For a representation theoretic interpretation of this ideal in terms of representation theory of Lie algebras see Stroppel~\cite{Strop}.  In 1992, Garsia and Procesi~\cite{GP} showed that the ring $R_{\mu} = \mathbb{Q}[x_1,\ldots,x_n]/I_\mu$ is isomorphic to the cohomology ring of a variety called the \textit{Springer variety associated to a nilpotent element} $X \in \mathfrak{N}(\mu)$.  Much work has been done to simplify the description of the Tanisaki ideal even further, including work by Biagioli, Faridi, and Rosas~\cite{Biag} in 2008.  Inspired by their work, we generalize the Tanisaki ideal in the author's thesis~\cite{Mb-thesis} and forthcoming joint work~\cite{Mb2} for a subclass of the family of varieties that naturally extends Springer varieties, called Hessenberg varieties.

\subsection{Definition of a Hessenberg variety}\label{subsec:Hess_Defn}
Hessenberg varieties were introduced by De Mari, Procesi, and Shayman~\cite{dMPS} in 1992.  Let $h$ be a map from $\{1,2,\ldots,n\}$ to itself.  Denote $h_i$ to be the image of $i$ under $h$.  An $n$-tuple $h=(h_1,\ldots,h_n)$ is a \textit{Hessenberg function} if it satisfies the two constraints:
\begin{eqnarray*}
(a) & i \leq h_i \leq n, & i\in\{1,\ldots,n\}\\
(b) & h_i \leq h_{i+1}, & i\in\{1,\ldots,n-1\}.
\end{eqnarray*}
A \textit{flag} is a nested sequence of $\mathbb{C}$-vector spaces $V_1\subseteq V_2\subseteq \cdots \subseteq V_n=\mathbb{C}^n$ where each $V_i$ has dimension $i$.  The collection of all such flags is called the \textit{full flag variety}  $\mathfrak{F}$.  Fix a nilpotent operator $X \in \Mat_n(\mathbb{C})$.  We define a \textit{Hessenberg variety} to be the following subvariety of the full flag variety:
$$\mathfrak{H}(X,h) = \{\mbox{Flags}\in\mathfrak{F} \;|\; X\cdot V_i\subseteq V_{h(i)} \mbox{ for all $i$} \}.$$
Since conjugating the nilpotent $X$ will produce a variety homeomorphic to $\mathfrak{H}(X,h)$ \cite[Proposition~2.7]{Tym}, we can assume that the nilpotent operator $X$ is in Jordan canonical form, with a weakly decreasing sequence of Jordan block sizes $\mu_1\geqslant\cdots\geqslant\mu_s>0$ so that $\sum_{i=1}^s \mu_i = n$.  We may view $\mu$ as a partition of $n$ or as a Young diagram with row lengths $\mu_i$.  Thus there is a one-to-one correspondence between Young diagrams and conjugacy classes of nilpotent operators.

For a fixed nilpotent operator $X$, there are two extremal cases for the choice of the Hessenberg function $h$: the minimal case occurs when $h(i)=i$ for all $i$, and the maximal case occurs when $h(i)=n$ for all $i$.  In the first case when $h=(1,2,\ldots,n)$, the variety $\mathfrak{H}(X,h)$ obtained is the Springer variety, which we denote $\mathfrak{S}_X$.  In the second case when $h=(n,\ldots,n)$, all flags satisfy the condition $X\cdot V_i\subseteq V_{h(i)}$ for all $i$ and hence $\mathfrak{H}(X,h)$ is the full flag variety $\mathfrak{F}$.

\subsection{Using \texorpdfstring{$\JT$}{(h,mu)}-fillings to compute the Betti numbers of Hessenberg varieties}\label{subsec:Tym_Work}
In 2005, Tymoczko~\cite{Tym} gave a combinatorial procedure for finding the dimensions of the graded parts of $H^*(\mathfrak{H}(X,h))$.  Let the Young diagram $\mu$ correspond to the Jordan canonical form of $X$ as given in Subsection~\ref{subsec:Hess_Defn}.  Any injective placing of the numbers $1,\ldots,n$ in a diagram $\mu$ with $n$ boxes is called a \textit{filling of $\mu$}.  It is called an \textit{(h-$\mu$)-filling} if it adheres to the following rule: a horizontal adjacency  \setlength{\unitlength}{.15in}\begin{picture}(2,1)(0,0)
\linethickness{.25pt}
\multiput(0,0)(0,1){2}{\line(1,0){2}}
\multiput(0,0)(1,0){3}{\line(0,1){1}}
\put(.5,.5){\makebox(0,0){\begin{small}$k$\end{small}}}
\put(1.5,.5){\makebox(0,0){\begin{small}$j$\end{small}}}\end{picture}
is allowed only if $k\leq h(j)$.  If $h$ and $\mu$ are clear from context, then we often call this a \textit{permissible filling}.  When $h=(3,3,3)$ all permissible fillings of $\mu=(2,1)$ coincide with all possible fillings as shown below.

\begin{figure}[!ht]
$$
\setlength{\unitlength}{.15in}\begin{picture}(2,2)(0,0)
\linethickness{.25pt}
\multiput(0,1)(0,1){2}{\line(1,0){2}}
\multiput(0,0)(1,0){2}{\line(0,1){2}}
\put(0,0){\line(1,0){1}}
\put(2,2){\line(0,-1){1}}
\put(.5,1.5){\makebox(0,0){1}}
\put(1.5,1.5){\makebox(0,0){2}}
\put(.5,.5){\makebox(0,0){3}}
\end{picture}, \;\;
\begin{picture}(2,2)(0,0)
\linethickness{.25pt}
\multiput(0,1)(0,1){2}{\line(1,0){2}}
\multiput(0,0)(1,0){2}{\line(0,1){2}}
\put(0,0){\line(1,0){1}}
\put(2,2){\line(0,-1){1}}
\put(.5,1.5){\makebox(0,0){1}}
\put(1.5,1.5){\makebox(0,0){3}}
\put(.5,.5){\makebox(0,0){2}}
\end{picture}, \;\;
\begin{picture}(2,2)(0,0)
\linethickness{.25pt}
\multiput(0,1)(0,1){2}{\line(1,0){2}}
\multiput(0,0)(1,0){2}{\line(0,1){2}}
\put(0,0){\line(1,0){1}}
\put(2,2){\line(0,-1){1}}
\put(.5,1.5){\makebox(0,0){2}}
\put(1.5,1.5){\makebox(0,0){3}}
\put(.5,.5){\makebox(0,0){1}}
\end{picture}, \;\;
\begin{picture}(2,2)(0,0)
\linethickness{.25pt}
\multiput(0,1)(0,1){2}{\line(1,0){2}}
\multiput(0,0)(1,0){2}{\line(0,1){2}}
\put(0,0){\line(1,0){1}}
\put(2,2){\line(0,-1){1}}
\put(.5,1.5){\makebox(0,0){2}}
\put(1.5,1.5){\makebox(0,0){1}}
\put(.5,.5){\makebox(0,0){3}}
\end{picture}, \;\;
\begin{picture}(2,2)(0,0)
\linethickness{.25pt}
\multiput(0,1)(0,1){2}{\line(1,0){2}}
\multiput(0,0)(1,0){2}{\line(0,1){2}}
\put(0,0){\line(1,0){1}}
\put(2,2){\line(0,-1){1}}
\put(.5,1.5){\makebox(0,0){3}}
\put(1.5,1.5){\makebox(0,0){1}}
\put(.5,.5){\makebox(0,0){2}}
\end{picture}, \mbox{ and  }
\begin{picture}(2,2)(0,0)
\linethickness{.25pt}
\multiput(0,1)(0,1){2}{\line(1,0){2}}
\multiput(0,0)(1,0){2}{\line(0,1){2}}
\put(0,0){\line(1,0){1}}
\put(2,2){\line(0,-1){1}}
\put(.5,1.5){\makebox(0,0){3}}
\put(1.5,1.5){\makebox(0,0){2}}
\put(.5,.5){\makebox(0,0){1}}
\end{picture}
$$
\caption{\label{fig:(333)-fillings_of_(2,1)}The six $\JT$-fillings for $h=(3,3,3)$ and $\mu=(2,1)$.}
\end{figure}

\noindent If $h=(1,3,3)$ then the fourth and fifth tableaux in Figure~\ref{fig:(333)-fillings_of_(2,1)} are not $\JT$-fillings since \begin{picture}(2,1)(0,0)
\linethickness{.25pt}
\multiput(0,0)(0,1){2}{\line(1,0){2}}
\multiput(0,0)(1,0){3}{\line(0,1){1}}
\put(.5,.5){\makebox(0,0){2}}
\put(1.5,.5){\makebox(0,0){1}}
\end{picture}
and \begin{picture}(2,1)(0,0)
\linethickness{.25pt}
\multiput(0,0)(0,1){2}{\line(1,0){2}}
\multiput(0,0)(1,0){3}{\line(0,1){1}}
\put(.5,.5){\makebox(0,0){3}}
\put(1.5,.5){\makebox(0,0){1}}
\end{picture}
are not allowable adjacencies for this $h$.

\begin{definition}[Dimension pair]\label{def:dimension_pair_(a,b)}
Let $h$ be a Hessenberg function and $\mu$ be a partition of $n$.  The pair $(a,b)$ is a \textit{dimension pair} of an $\JT$-filling $T$ if 
\begin{enumerate}
\item $b>a$,
\item $b$ is below $a$ and in the same column, or $b$ is in any column strictly to the left of $a$, and
\item if some box with filling $c$ happens to be adjacent and to the right of $a$, then $b\leq h(c)$.
\end{enumerate}
\end{definition}

\begin{theorem}[Tymoczko]~\cite[Theorem 1.1]{Tym}
The dimension of $H^{2k}(\mathfrak{H}(X,h))$ is the number of $\JT$-fillings $T$ such that $T$ has $k$ dimension pairs.
\end{theorem}

\begin{remark}
Tymoczko proves this theorem by providing an explicit geometric construction which partitions $\mathfrak{H}(X,h)$ into pieces homeomorphic to complex affine space.  In fact, this is a paving by affines and consequently determines the Betti numbers of $\mathfrak{H}(X,h)$.  See \cite{Tym} for precise details.
\end{remark}

\begin{example}
Fix $h=(1,3,3)$ and let $\mu$ have shape $(2,1)$.  Figure~\ref{fig:four_JT_fillings} gives all possible $\JT$-fillings and their corresponding dimension pairs.  We conclude $H^{0}$ has dimension 1 since exactly one filling has 0 dimension pairs.  $H^{2}$ has dimension 2 since exactly two fillings have 1 dimension pair each.  Lastly, $H^{4}$ has dimension 1 since the remaining filling has 2 dimension pairs.

\begin{figure}[!ht]
$$
\begin{array}{rclcrcl}
\begin{picture}(2,2)(0,0)
\linethickness{.25pt}
\multiput(0,1)(0,1){2}{\line(1,0){2}}
\multiput(0,0)(1,0){2}{\line(0,1){2}}
\put(0,0){\line(1,0){1}}
\put(2,2){\line(0,-1){1}}
\put(.5,1.5){\makebox(0,0){1}}
\put(1.5,1.5){\makebox(0,0){2}}
\put(.5,.5){\makebox(0,0){3}}
\end{picture} & \longleftrightarrow & (1,3),(2,3) & &
\begin{picture}(2,2)(0,0)
\linethickness{.25pt}
\multiput(0,1)(0,1){2}{\line(1,0){2}}
\multiput(0,0)(1,0){2}{\line(0,1){2}}
\put(0,0){\line(1,0){1}}
\put(2,2){\line(0,-1){1}}
\put(.5,1.5){\makebox(0,0){1}}
\put(1.5,1.5){\makebox(0,0){3}}
\put(.5,.5){\makebox(0,0){2}}
\end{picture} & \longleftrightarrow & (1,2)\\
\begin{picture}(2,2)(0,0)
\linethickness{.25pt}
\multiput(0,1)(0,1){2}{\line(1,0){2}}
\multiput(0,0)(1,0){2}{\line(0,1){2}}
\put(0,0){\line(1,0){1}}
\put(2,2){\line(0,-1){1}}
\put(.5,1.5){\makebox(0,0){2}}
\put(1.5,1.5){\makebox(0,0){3}}
\put(.5,.5){\makebox(0,0){1}}
\end{picture} &\longleftrightarrow& \mbox{no dimension pairs} & &
\begin{picture}(2,2)(0,0)
\linethickness{.25pt}
\multiput(0,1)(0,1){2}{\line(1,0){2}}
\multiput(0,0)(1,0){2}{\line(0,1){2}}
\put(0,0){\line(1,0){1}}
\put(2,2){\line(0,-1){1}}
\put(.5,1.5){\makebox(0,0){3}}
\put(1.5,1.5){\makebox(0,0){2}}
\put(.5,.5){\makebox(0,0){1}}
\end{picture} & \longleftrightarrow & (2,3)
\end{array}
$$
\caption{\label{fig:four_JT_fillings}The four $\JT$-fillings for $h=(1,3,3)$ and $\mu=(2,1)$.}
\end{figure}
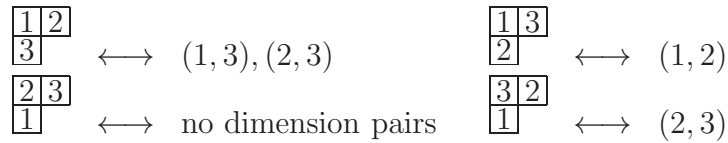
\end{example}

\subsection{The map \texorpdfstring{$\Phi$}{Phi} from \texorpdfstring{$\JT$}{(h,mu)}-fillings to monomials \texorpdfstring{$\Ah$}{A-h(mu)}}\label{subsec:Phi_map}
Let $R$ be the polynomial ring $\mathbb{Z}[x_1,\ldots,x_n]$.  We introduce a map from $\JT$-fillings onto a set of monomials in $R$.   First, we provide some notation for the set of dimension pairs.

\begin{definition}[The set $\D$ of dimension pairs of $T$]\label{def:dimension_pairs_set}
Fix a partition $\mu$ of $n$.  Let $T$ be an $\JT$-filling.  Define $\D$ to be the set of dimension pairs of $T$ according to Subsection~\ref{subsec:Tym_Work}. For a fixed $y \in \{2,\ldots,n\}$, define
$$\D_y := \left\{ (x,y) \;|\; (x,y) \in \D \right\}.$$
The number of dimension pairs of an $\JT$-filling $T$ is called the \textit{dimension of T}.
\end{definition}

Fix a Hessenberg function $h$ and a partition $\mu$ of $n$.  The map $\Phi$ is the following:
$$\Phi: \{\JT\mbox{-fillings}\} \longrightarrow R \mbox{\;\;\;\;\; defined by\;\;\;\;\;} T \longmapsto \prod\limits_{\substack{(i,j) \in \D_j \\ 2 \leq j \leq n}} x_j.$$
Denote the image of $\Phi$ by $\Ah$.  By abuse of notation we also denote the $\mathbb{Q}$-linear span of these monomials by $\Ah$.  Denote the formal $\mathbb{Q}$-linear span of the $\JT$-fillings by $M^{h,\mu}$.  Extending $\Phi$ linearly, we get a map on vector spaces $\Phi: M^{h,\mu} \rightarrow \Ah$.

\begin{remark}\label{rem:aBa-monoms_have_no_1}
Any monomial $\X \in \Ah$ will be of the form $x_2^{\alpha_2} \cdots x_n^{\alpha_n}$.  That is, the variable $x_1$ can never appear in $\X$ since 1 will never be the larger number in a dimension pair.
\end{remark}

\begin{theorem}\label{thm:JT-fillings_to_Monomials}
If $\mu$ is a partition of $n$, then $\Phi$ is a well-defined degree-preserving map from a set of $\JT$-fillings onto monomials $\Ah$.  That is, $r$-dimensional $\JT$-fillings map to degree-$r$ monomials in $\Ah$.
\end{theorem}
\begin{proof}
Let $T$ be an $\JT$-filling of dimension $r$.  Then $T$ has $r$ dimension pairs by definition.  By construction $\Phi(T)$ will have degree $r$.  Hence the map is degree-preserving.
\end{proof}

%%%%%%%%%%%%%%%%%%%%%%%%%%%%%%%%%%%%%%%%%%%%%%%%%%%%%%
%%%%%%%%%%%%%%%%%%%%%%%%%%%%%%%%%%%%%%%%%%%%%%%%%%%%%%
%%%%%%%%%%%%%%%%%%%%%%%%%%%%%%%%%%%%%%%%%%%%%%%%%%%%%%
%%%%%%%%%%%%%%%%%%%%%%%%%%%%%%%%%%%%%%%%%%%%%%%%%%%%%%
%%%%%%%%%%%%%%%%%%%%%%%%%%%%%%%%%%%%%%%%%%%%%%%%%%%%%%
%%%%%%%%%%%%%%%%%%%%%%%%%%%%%%%%%%%%%%%%%%%%%%%%%%%%%%
%%%%%%%%%%%%%%%%%%%%%%%%%%%%%%%%%%%%%%%%%%%%%%%%%%%%%%
%%%%%%%%%%%%%%%%%%%%%%%%%%%%%%%%%%%%%%%%%%%%%%%%%%%%%%
%%%%%%%%%%%%%%%%%% S E C T I O N  2 %%%%%%%%%%%%%%%%%%

\section{The Springer setting}\label{sec:Springer_Setting}
In this section we will fill in the details of Figure~\ref{fig:The_Springer_Triangle}.  Recall that if we fix the Hessenberg function $h=(1,2,\ldots,n)$ and let the nilpotent operator $X$ (equivalently, the shape $\mu$) vary, the Hessenberg variety $\mathfrak{H}(X,h)$ obtained is the Springer variety $\mathfrak{S}_X$.  Since this section focuses on this setting, we omit $h$ in our notation.  For instance, the image of $\Phi$ is $\A$.  Similarly, the Garsia-Procesi basis will be denoted $\B$ (as it is denoted in the literature~\cite{GP}).

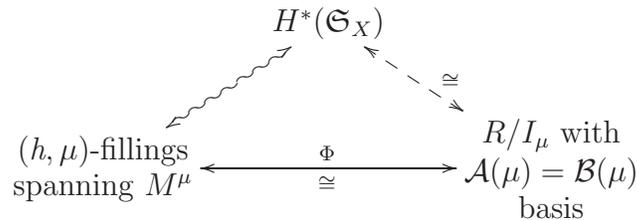
\begin{figure}[!ht]
$$
\xymatrix{
& H^*(\mathfrak{S}_X) \ar@{<~>}[dl] \ar@{<-->}[dr]^{\cong}\\
\txt{$\JT$-fillings\\spanning $M^{\mu}$}\ar@{<->}[rr]_{\cong}^{\Phi} & & \txt{$R/I_{\mu}$ with\\$\A=\B$\\basis}
}
$$
%\vspace{-.25in}
\caption{\label{fig:The_Springer_Triangle}Springer setting.}
\end{figure}

In Subsection~\ref{subsec:Remarks_on_Phi}, we recast the statement of the graded vector space morphism $\Phi$ to the setting of Springer varieties.  In Subsection~\ref{subsec:Psi_map}, we define an inverse map $\Psi$ from the span of monomials $\A$ to the formal linear span of $\JT$-fillings, thereby giving not only a bijection of sets but also a graded vector space isomorphism.  We prove that $\Psi$ is an isomorphism in Corollary~\ref{cor:A_and_M_isom_vect_spaces}.  This completes the bottom leg of the triangle in Figure~\ref{fig:The_Springer_Triangle}.  In Subsection~\ref{subsec:A_equals_B_monomials}, we modify the work of Garsia and Procesi~\cite{GP} and develop a technique to build the $\JT$-filling corresponding to a monomial in their quotient basis $\B$.  We conclude $\A = \B$.

\subsection{Remarks on the map \texorpdfstring{$\Phi$}{Phi} when \texorpdfstring{$h=(1,2,\ldots,n)$}{h=(1,2,...,n)}}\label{subsec:Remarks_on_Phi}
Fix a partition $\mu$ of $n$.  Upon considering the combinatorial rules governing a permissible filling of a Young diagram, we see that if $h=(1,2,\ldots,n)$, then the $\JT$-fillings are just the row-strict tableaux of shape $\mu$.  Suppressing $h$, we denote the formal linear span of these tableaux by $M^{\mu}$.  This is the standard symbol for this space, commonly known as the permutation module corresponding to $\mu$ (see expository work of Fulton~\cite{F}).  In this specialized setting, the map $\Phi$ is simply
$$\Phi: M^\mu \relbar\joinrel\twoheadrightarrow \A \mbox{\;\;\;\;\; defined by\;\;\;\;\;} T \longmapsto \prod\limits_{\substack{(i,j) \in \D_j \\ 2 \leq j \leq n}} x_j,$$

\noindent and hence Theorem~\ref{thm:JT-fillings_to_Monomials} specializes to the following.
\begin{theorem}\label{thm:Tab_to_GP}
If $\mu$ is a partition of $n$, then $\Phi$ is a well-defined degree-preserving map from the set of row-strict tableaux in $M^\mu$ onto the monomials $\A$.  That is, $r$-dimensional tableaux in $M^\mu$ map to degree-$r$ monomials in $\A$.
\end{theorem}

\begin{example}\label{exam:tabloid_to_monomial}
Let $\mu = (2,2,2)$ have the filling $T=$ \setlength{\unitlength}{.15in}
\begin{picture}(2,3)(0,0)
\linethickness{.25pt}
\put(0,0){\line(1,0){2}}
\put(0,1){\line(1,0){2}}
\put(0,2){\line(1,0){2}}
\put(0,3){\line(1,0){2}}
\put(0,3){\line(0,-1){3}}
\put(1,3){\line(0,-1){3}}
\put(2,3){\line(0,-1){3}}
\put(.5,.5){\makebox(0,0){4}}
\put(1.5,.5){\makebox(0,0){5}}
\put(.5,1.5){\makebox(0,0){3}}
\put(1.5,1.5){\makebox(0,0){6}}
\put(.5,2.5){\makebox(0,0){1}}
\put(1.5,2.5){\makebox(0,0){2}}
\end{picture} \;. Suppressing the commas for ease of viewing, the contributing dimension pairs are (23), (24), (25), (26) and (34).  Observe $(23) \in \D_3$, $(24),(34) \in \D_4$, $(25) \in \D_5$, and $(26) \in \D_6$.  Hence $\Phi$ takes this tableau to the monomial $x_3 x_4^2 x_5 x_6 \in \A$.
\end{example}

In the next subsection we will give an explicit algorithm to recover the original row-strict tableau from any monomial in $\A$.  In particular, Example~\ref{exam:inverse_map_for_three_row} applies the inverse algorithm to the example above.

\subsection{The inverse map \texorpdfstring{$\Psi$}{Psi} from monomials in \texorpdfstring{$\A$}{A(mu)} to \texorpdfstring{$\JT$}{(h,mu)}-fillings}\label{subsec:Psi_map}

The map back from a monomial $\X\in\A$ to an $\JT$-filling is not as transparent.  We will construct the tableau by filling it in reverse order starting with the number $n$.  The next definitions give us the language to speak about where we can place $n$ and the subsequent numbers.

\begin{definition}[Composition of $n$] Let $\rho$ be a partition of $n$ corresponding to a diagram of shape $(\rho_1, \rho_2, \ldots, \rho_s)$ that need not be a proper Young diagram.  That is, the sequence neither has to weakly increase nor decrease and some $\rho_i$ may even be zero.  An ordered partition of this kind is often called a \textit{composition} of $n$ and is denoted $\rho\vDash n$.
\end{definition}

\begin{definition}[Dimension-ordering of a composition]\label{def:dimension_ordering}
We define a \textit{dimension-ordering} of certain boxes in a composition $\rho$ in the following manner.  Order the boxes on the far-right of each row starting from the rightmost column to the leftmost column going from top to bottom in the columns containing more than one far-right box.
\end{definition}

\begin{example} If $\rho=(2,1,0,3,4)\vDash12$, then the ordering is

\medskip
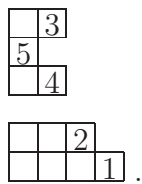
\begin{center}
\setlength{\unitlength}{.15in}
\begin{picture}(4,6)(0,0)
\linethickness{.25pt}
\put(0,3){\line(1,0){2}}
\put(0,4){\line(1,0){2}}
\put(0,5){\line(1,0){2}}
\put(0,6){\line(1,0){2}}
\put(0,6){\line(0,-1){3}}
\put(1,6){\line(0,-1){3}}
\put(2,6){\line(0,-1){1}}
\put(2,4){\line(0,-1){1}}
\put(1.5,5.45){\makebox(0,0){3}}
\put(.5,4.45){\makebox(0,0){5}}
\put(1.5,3.425){\makebox(0,0){4}}
\multiput(0,0)(1,0){4}{\line(0,1){2}}
\put(4,0){\line(0,1){1}}
\put(0,0){\line(1,0){4}}
\put(0,1){\line(1,0){4}}
\put(0,2){\line(1,0){3}}
\put(3.5,.425){\makebox(0,0){1}}
\put(2.5,1.45){\makebox(0,0){2}}
\end{picture}~.
\end{center}
Notice that imposing a dimension-ordering on a diagram places exactly one number in the far-right box of each non-empty row.
\end{example}

\begin{definition}[Subfillings and subdiagrams of a composition]\label{def:subfilling}
Let $T$ be a filling of a composition $\rho$ of $n$.  If the values $i+1, i+2,\ldots, n$ and their corresponding boxes are removed from $T$, then what remains is called a \textit{subfilling  of $T$} and is denoted $T^{(i)}$.  Ignoring the numbers in these remaining $i$ boxes, the shape is called a \textit{subdiagram of $\rho$} and is denoted $\rho^{(i)}$.
\end{definition}

Observe that $\rho^{(i)}$ need no longer be a composition.  For example, let $\rho=$ \setlength{\unitlength}{.15in}\begin{picture}(3,1)(0,0) \linethickness{.25pt}\multiput(0,0)(0,1){2}{\line(1,0){3}}
\multiput(0,0)(1,0){4}{\line(0,1){1}}
\end{picture} have the filling $T=$ \setlength{\unitlength}{.15in}\begin{picture}(3,1)(0,0) \linethickness{.25pt}\multiput(0,0)(0,1){2}{\line(1,0){3}}
\multiput(0,0)(1,0){4}{\line(0,1){1}}
\put(.5,.45){\makebox(0,0){1}}\put(1.5,.45){\makebox(0,0){3}}\put(2.5,.45){\makebox(0,0){2}}
\end{picture}.  Then $T^{(2)}$ is \begin{picture}(3,1)(0,0) \multiput(0,0)(0,1){2}{\line(1,0){1}}
\put(2,0){\line(1,0){1}} \put(2,1){\line(1,0){1}}
\multiput(0,0)(1,0){4}{\line(0,1){1}}
\put(.5,.45){\makebox(0,0){1}}\put(2.5,.45){\makebox(0,0){2}}
\end{picture} and so $\rho^{(2)}$ gives the subdiagram \begin{picture}(3,1)(0,0) \multiput(0,0)(0,1){2}{\line(1,0){1}}
\put(2,0){\line(1,0){1}} \put(2,1){\line(1,0){1}}
\multiput(0,0)(1,0){4}{\line(0,1){1}}
\end{picture} which is not a composition.  The next property gives a sufficient condition on $T$ to ensure $\rho^{(i)}$ is a composition.

\begin{property}
A filling $T$ of a composition $\rho$ of $n$ satisfies the \textit{subfilling property} if the number $i$ is in the rightmost box of some row of the subfilling $T^{(i)}$ for each $i \in \{1,\ldots,n\}$.
\end{property}

\begin{lemma}\label{lem:dim_ordered_filling_implies_row_strict} Let $T$ be a filling of a composition $\rho$ of $n$.  Then the following are equivalent:
\vspace{-.05in}
\begin{center}
\begin{tabular}{c l}
(a) & $T$ satisfies the Subfilling Property.\\
(b) & $T$ is a row-strict filling of $\rho$.
\end{tabular}\end{center}

\noindent In particular if the composition $\rho$ is a Young diagram satisfying the Subfilling Property, then $T$ lies in $M^{\rho}$.
\end{lemma}

\begin{proof}
Let $T$ be a filling of composition $\rho$ of $n$.  Suppose $T$ is not row-strict.  Then there exists some row in $\rho$ with an adjacent filling of two numbers \setlength{\unitlength}{.15in}\begin{picture}(2,1)(0,0)
\linethickness{.25pt}
\multiput(0,0)(0,1){2}{\line(1,0){2}}
\multiput(0,0)(1,0){3}{\line(0,1){1}}
\put(.5,.5){\makebox(0,0){\begin{small}$k$\end{small}}}
\put(1.5,.5){\makebox(0,0){\begin{small}$j$\end{small}}}\end{picture}
such that $k>j$.  However the subfilling $T^{(k)}$ does not have $k$ in the rightmost box of this row, so $T$ does not satisfy the Subfilling Property.  Hence $(a)$ implies $(b)$.  For the converse, suppose $T$ does not satisfy the Subfilling Property.  Then there exists a number $i$ such that $i$ is not in the far-right box of some nonzero row in $T^{(i)}$.  Thus there is some $k$ in this row that is smaller and to the right of $i$ so $T$ is not row-strict.  Hence $(b)$ implies $(a)$.
\end{proof}

\begin{lemma}\label{lem:dim_order_filling_gives_these_many_dimpairs} Let $\rho = (\rho_1,\rho_2,\ldots,\rho_s)$ be a composition of $n$.  Suppose that $r$ of the $s$ entries $\rho_i$ are nonzero.  We claim:
\begin{enumerate}[(a)]
\item There exist exactly $r$ positions where $n$ can be placed in a row-strict composition.
\item Let $T$ be a row-strict filling of $\rho$.  If $n$ is placed in the box of $T$ with dimension-ordering $i$ in $\{1,\ldots,r\}$, then $n$ is in a dimension pair with exactly $i-1$ other numbers; that is, $|\DP^T_n| = i-1$.
\end{enumerate}
\end{lemma}

\begin{proof}
Suppose $\rho = (\rho_1,\rho_2,\ldots,\rho_s)$ is a composition of $n$ where $r$ of the $s$ entries are nonzero.  Claim $(a)$ follows by the definition of row-strict and the fact that $n$ is the largest number in any filling of $\rho\vDash n$.  To illustrate the proof of $(b)$, consider the following schematic for $\rho$:

\medskip
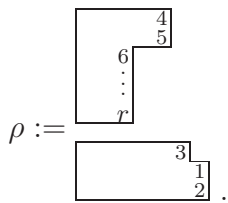
\begin{center}
\setlength{\unitlength}{.1in}
\begin{picture}(7,10)(0,0)
\linethickness{.25pt}
\put(0,0){\line(1,0){7}}
\put(0,3){\line(1,0){6}}
\put(0,4){\line(1,0){3}}
\put(0,10){\line(1,0){5}}
\put(3,8){\line(1,0){2}}
\put(0,3){\line(0,-1){3}}
\put(6,3){\line(0,-1){1}}
\put(6,2){\line(1,0){1}}
\put(7,2){\line(0,-1){2}}
\put(0,10){\line(0,-1){6}}
\put(5,10){\line(0,-1){2}}
\put(3,8){\line(0,-1){4}}
\put(6.5,1.5){\makebox(0,0){\begin{scriptsize}1\end{scriptsize}}}
\put(6.5,.5){\makebox(0,0){\begin{scriptsize}2\end{scriptsize}}}
\put(5.5,2.5){\makebox(0,0){\begin{scriptsize}3\end{scriptsize}}}
\put(4.5,9.5){\makebox(0,0){\begin{scriptsize}4\end{scriptsize}}}
\put(4.5,8.5){\makebox(0,0){\begin{scriptsize}5\end{scriptsize}}}
\put(2.5,7.5){\makebox(0,0){\begin{scriptsize}6\end{scriptsize}}}
\put(2.5,6.5){\makebox(0,0){\begin{scriptsize}\vdots\end{scriptsize}}}
\put(2.5,4.4){\makebox(0,0){\begin{footnotesize}$r$\end{footnotesize}}}
\put(-2,3.5){\makebox(0,0){$\rho:=$}}
\end{picture}~.
\end{center}

\noindent Enumerate the far-right boxes of each nonempty row so that they are dimension-ordered as in the schematic above.  Let $T$ be a row-strict filling of $\rho$.  Suppose $n$ lies in the box with dimension-ordering $i \in \{1,\ldots,r\}$.  It suffices to count the number of dimension pairs with $n$, or simply $|\DP^T_n|$ since $n$ is the largest value in the filling.  Thus we want to count the distinct values $\beta$ such that $(\beta,n)\in \DP^T_n$.  We need not concern ourselves with boxes with values $\beta$ in the same column below or anywhere left of the $i^{th}$ dimension-ordered box for if such a $\beta$ had $(\beta,n)\in \DP^T_n$, then that would imply $\beta>n$ which is impossible (see $\bullet$-shaded boxes in figure below).  We also need not concern ourselves with any boxes that are in the same column above or anywhere to the right of the $i^{th}$ dimension-ordered box if it has a neighbor $j$ immediately right of it (see $\circ$-shaded boxes in figure below).

\medskip
\begin{center}
\setlength{\unitlength}{.1in}
\begin{picture}(7,10)(0,0)
\linethickness{.25pt}
\put(0,0){\line(1,0){7}}
\put(0,3){\line(1,0){6}}
\put(0,4){\line(1,0){3}}
\put(0,10){\line(1,0){5}}
\put(3,8){\line(1,0){2}}
\put(0,3){\line(0,-1){3}}
\put(6,3){\line(0,-1){1}}
\put(6,2){\line(1,0){1}}
\put(7,2){\line(0,-1){2}}
\put(0,10){\line(0,-1){6}}
\put(5,10){\line(0,-1){2}}
\put(3,8){\line(0,-1){4}}
\put(6.5,1.5){\makebox(0,0){\begin{scriptsize}1\end{scriptsize}}}
\put(6.5,.5){\makebox(0,0){\begin{scriptsize}2\end{scriptsize}}}
\put(5.5,2.5){\makebox(0,0){\begin{scriptsize}3\end{scriptsize}}}
\put(4.5,9.5){\makebox(0,0){\begin{scriptsize}4\end{scriptsize}}}
\put(4.5,8.5){\makebox(0,0){\begin{scriptsize}5\end{scriptsize}}}
\put(2.5,7.5){\makebox(0,0){\begin{scriptsize}6\end{scriptsize}}}
\multiput(2.3,6)(0,.45){2}{$\cdot$}
\put(2.5,5.5){\makebox(0,0){\begin{footnotesize}$i$\end{footnotesize}}}
\put(-2,3.5){\makebox(0,0){$\rho:=$}}
\multiput(.1,.15)(0,1){3}{$\bullet$}
\multiput(.1,4.15)(0,1){1}{$\bullet$}
\multiput(.3,5.15)(0,.55){3}{$\cdot$}
\multiput(.1,7.15)(0,1){3}{$\bullet$}
\multiput(1.1,.15)(0,1){3}{$\bullet$}
\multiput(1.1,4.15)(0,1){1}{$\bullet$}
\multiput(1.3,5.15)(0,.55){3}{$\cdot$}
\multiput(1.1,7.15)(0,1){3}{$\bullet$}
\multiput(2.1,.15)(0,1){3}{$\bullet$}
\multiput(2.1,4.15)(0,1){1}{$\bullet$}
\multiput(3.1,.15)(0,1){3}{$\circ$}
\multiput(4.1,.15)(0,1){3}{$\circ$}
\multiput(3.1,8.15)(0,1){2}{$\circ$}
\multiput(2.1,8.15)(0,1){2}{$\circ$}
\multiput(5.1,.15)(0,1){2}{$\circ$}
\end{picture}~.
\end{center}

\noindent If $\beta$ were in such a box, then $(\beta,n)\in \DP^T_n$ would imply $n\leq h(j)$ which is impossible since $h(j)=j$ and $j<n$.  That leaves exactly the $i-1$ boxes which are dimension-ordered boxes that are in the same column above $n$ or anywhere to the right of $n$, each of which are by definition in $\DP^T_n$.  Hence $|\DP^T_n| = i-1$ and $(b)$ is shown.
\end{proof}

\begin{lemma}\label{lem:CompositionDP_equals_TabloidDP}
Suppose $T$ is a row-strict filling of a composition $\rho$ of $n$.  If $i \in \{1,\ldots,n\}$, then $|\DP^{T^{(i)}}_i| = |\D_i|$.
\end{lemma}

\begin{proof}
Consider the subfilling $T^{(i)}$.  All the existing pairs $(\beta,i) \in \DP^{T^{(i)}}_i$ will still be valid dimension pairs in $T$ if we restore the numbers $i+1,\ldots,n$ and their corresponding boxes.  Hence the inequality $|\DP^{T^{(i)}}_i| \leq |\D_i|$ holds.  However no further pairs $(\beta,i)$ with $\beta < i$ can be created by restoring numbers larger than $i$.  Thus we get equality.
\end{proof}

\begin{lemma}\label{lem:why_this_box_exists}
Fix a partition $\mu$ of $n$.  Let $T$ be a tableau in $M^{\mu}$.  Suppose $\Phi(T)=\X$.  For each $i\in\{2,\ldots,n\}$, consider the subdiagram $\mu^{(i)}$ of $\mu$ corresponding to the subfilling $T^{(i)}$ of $T$.  Then each $\mu^{(i)}$ has at least $\alpha_i+1$ nonzero rows where $\alpha_i$ is the exponent of $x_i$ in the monomial $\X$.
\end{lemma}

\begin{proof}
Fix a partition $\mu$ of $n$.  Let $T\in M^{\mu}$ and $\X\in\A$ be $\Phi(T)$.  By Remark~\ref{rem:aBa-monoms_have_no_1}, $\X$ is of the form $x_2^{\alpha_2}\cdots x_n^{\alpha_n}$.  Suppose that the claim does not hold.  Then there is some $i\in\{2,\ldots,n\}$ for which $\mu^{(i)}$ has $r$ nonzero rows and $r<\alpha_i+1$.  Lemma~\ref{lem:dim_order_filling_gives_these_many_dimpairs} implies the number of dimension pairs in $\DP^{T^{(i)}}_i$ is at most $r-1$.  Thus $|\DP^{T^{(i)}}_i| \leq r-1 < \alpha_i$.  Since $|\DP^{T^{(i)}}_i| = |\D_i|$ by Lemma~\ref{lem:CompositionDP_equals_TabloidDP}, it follows that $|\D_i|<\alpha_i$, contradicting the fact that variable $x_i$ has exponent $\alpha_i$.
\end{proof}

\begin{theorem}[A map from $\A$ to $\JT$-fillings]\label{thm:GP_to_Tab}
Given a partition $\mu$ of $n$, there exists a well-defined dimension-preserving map $\Psi$ from the monomials $\A$ to the set of row-strict tableaux in $M^\mu$.  That is, $\Psi$ maps degree-$r$ monomials in $\A$ to $r$-dimensional $\JT$-fillings in $M^\mu$.  Moreover the composition $\A \stackrel{\Psi}{\longrightarrow} M^\mu \stackrel{\Phi}{\longrightarrow} \A$ is the identity.
\end{theorem}
\begin{proof}
Fix a partition $\mu = (\mu_1,\mu_2,\ldots,\mu_k)$ of $n$.  Let $\X$ be a degree-$r$ monomial in $\A$.  Remark~\ref{rem:aBa-monoms_have_no_1} reminds us that $\X$ is of the form $x_2^{\alpha_2} \cdots x_n^{\alpha_n}$.  The goal is to construct a map $\Psi$ from $\A$ to $M^\mu$ such that $\Psi(\X)$ is an $r$-dimensional tableau in $M^\mu$ and $(\Phi \circ \Psi)(\X)=\X$.  Recall $\A$ is the image of $M^{\mu}$ under $\Phi$ so we know there exists some tableau $T'\in M^{\mu}$ with $|\DP^{T'}_i| = \alpha_i$ for each $i\in\{2,\ldots,n\}$.

We now construct a filling $T$ (not a priori the same as $T'$) by giving $\mu$ a precise row-strict filling to be described next.  To construct $T$ we iterate the algorithm below with a triple-datum of the form $(\mu^{(i)}, \;i, \;x_i^{\alpha_i})$ of a composition $\mu^{(i)}$ of $i$, an integer $i$, and the $x_i^{\alpha_i}$-part of $\X$.  Start with $i=n$ in which case $\mu^{(n)}$ is $\mu$ itself; then decrease $i$ by one each time and repeat the steps below with the new triple-datum.  The algorithm is as follows:

\begin{enumerate}
\itemsep -.5ex
\item Input the triple-datum.
\item Impose the dimension-ordering on the rightmost boxes of $\mu^{(i)}$.
\item Place $i$ in the box with dimension-order $\alpha_i+1$.
\item If $i\geq2$, then remove the box with the entry $i$ to get a new subdiagram $\mu^{(i-1)}$.  Pass the new triple-datum $(\mu^{(i-1)}, \;i-1, \;x_{i-1}^{\alpha_{i-1}})$ to Step 1.
\item If $i=1$, then the final number 1 is forced in the last remaining box.  Replace all $n-1$ removed numbers and call this tableau $T$.
\end{enumerate}

We confirm that this algorithm is well-defined and produces a tableau in $M^{\mu}$.  Step 3 can be performed because Lemma~\ref{lem:why_this_box_exists} ensures the box exists.  The Subfilling Property ensures that the subdiagram at Step 4 is indeed a composition.  By Lemma~\ref{lem:dim_ordered_filling_implies_row_strict}, $T$ is row-strict and hence lies in $M^{\mu}$.

We are left to show $\Phi$ maps $T$ to the original $\X\in\A$ from which we started.  It suffices to check that if the exponent of $x_i$ in $\X$ is $\alpha_i$, then $|\DP^T_i| = \alpha_i$ for each $i\in\{2,\ldots,n\}$.  By Lemma~\ref{lem:dim_order_filling_gives_these_many_dimpairs}, when $i=n$ we know $|\DP^T_n| = \alpha_n$.  At each iteration after this initial step, we remove one more box from $\mu$.  At step $i=m$ for $m<n$, we placed $m$ into $\mu^{(m)}$ in the box with dimension-order $\alpha_m+1$.  Hence $|\DP^{T^{(m)}}_m| = \alpha_m$ by Lemma~\ref{lem:dim_order_filling_gives_these_many_dimpairs}.  But $|\DP^{T^{(m)}}_m| = |\D_m|$ by Lemma~\ref{lem:CompositionDP_equals_TabloidDP}.  Thus $|\D_m| = \alpha_m$ as desired.  Hence given the monomial $\X \in \A$, we see by construction of $T=\Psi(\X)$ that $T$ has the desired dimension pairs to map back to $\X$ via the map $\Phi$.  That is, the composition $\Phi\circ\Psi$ is the identity on $\A$.
\end{proof}

\begin{example}\label{exam:inverse_map_for_three_row}
Let $\mu=(2,2,2)$ and consider the monomial $x_3 x_4^2 x_5 x_6$ from Example~\ref{exam:tabloid_to_monomial}.  We show that this monomial will map to the filling

$$\setlength{\unitlength}{.15in}
\begin{picture}(2,3)(0,0)
\linethickness{.25pt}
\put(0,0){\line(1,0){2}}
\put(0,1){\line(1,0){2}}
\put(0,2){\line(1,0){2}}
\put(0,3){\line(1,0){2}}
\put(0,3){\line(0,-1){3}}
\put(1,3){\line(0,-1){3}}
\put(2,3){\line(0,-1){3}}
\put(.5,.5){\makebox(0,0){4}}
\put(1.5,.5){\makebox(0,0){5}}
\put(.5,1.5){\makebox(0,0){3}}
\put(1.5,1.5){\makebox(0,0){6}}
\put(.5,2.5){\makebox(0,0){1}}
\put(1.5,2.5){\makebox(0,0){2}}
\end{picture}$$

\noindent which we showed in Example~\ref{exam:tabloid_to_monomial} maps to the monomial $x_3 x_4^2 x_5 x_6$ under $\Phi$.  For clarity in the following flowchart below, we label the dimension-ordered boxes at each stage in small font with letters $a,b,c$ to mean 1st, 2nd, 3rd dimension-ordered boxes respectively.  Place 6 in the second dimension-ordered box $b$ since the exponent of $x_6$ is 1.  Place 5 in the second dimension-ordered box $b$ since the exponent of $x_5$ is 1.  Place 4 in the third dimension-ordered box $c$ since the exponent of $x_4$ is 2.  And so on.

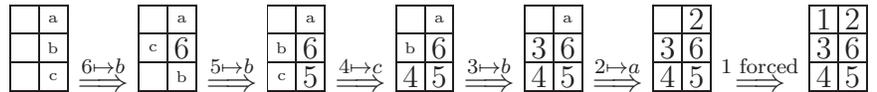
\begin{center}
\setlength{\unitlength}{.15in}
\begin{picture}(2,3)(0,0)
\linethickness{.25pt}
\put(0,0){\line(1,0){2}}
\put(0,1){\line(1,0){2}}
\put(0,2){\line(1,0){2}}
\put(0,3){\line(1,0){2}}
\put(0,3){\line(0,-1){3}}
\put(1,3){\line(0,-1){3}}
\put(2,3){\line(0,-1){3}}
\put(1.5,.5){\makebox(0,0){\begin{tiny}c\end{tiny}}}
\put(1.5,1.5){\makebox(0,0){\begin{tiny}b\end{tiny}}}
\put(1.5,2.5){\makebox(0,0){\begin{tiny}a\end{tiny}}}
\end{picture}
$\stackrel{6\mapsto b}{\Longrightarrow}$
\begin{picture}(2,3)(0,0)
\linethickness{.25pt}
\put(0,0){\line(1,0){2}}
\put(0,1){\line(1,0){2}}
\put(0,2){\line(1,0){2}}
\put(0,3){\line(1,0){2}}
\put(0,3){\line(0,-1){3}}
\put(1,3){\line(0,-1){3}}
\put(2,3){\line(0,-1){3}}
\put(1.5,1.5){\makebox(0,0){6}}
\put(1.5,.5){\makebox(0,0){\begin{tiny}b\end{tiny}}}
\put(.5,1.5){\makebox(0,0){\begin{tiny}c\end{tiny}}}
\put(1.5,2.5){\makebox(0,0){\begin{tiny}a\end{tiny}}}
\end{picture}
$\stackrel{5\mapsto b}{\Longrightarrow}$
\begin{picture}(2,3)(0,0)
\linethickness{.25pt}
\put(0,0){\line(1,0){2}}
\put(0,1){\line(1,0){2}}
\put(0,2){\line(1,0){2}}
\put(0,3){\line(1,0){2}}
\put(0,3){\line(0,-1){3}}
\put(1,3){\line(0,-1){3}}
\put(2,3){\line(0,-1){3}}
\put(1.5,.5){\makebox(0,0){5}}
\put(1.5,1.5){\makebox(0,0){6}}
\put(.5,1.5){\makebox(0,0){\begin{tiny}b\end{tiny}}}
\put(.5,.5){\makebox(0,0){\begin{tiny}c\end{tiny}}}
\put(1.5,2.5){\makebox(0,0){\begin{tiny}a\end{tiny}}}
\end{picture}
$\stackrel{4\mapsto c}{\Longrightarrow}$
\begin{picture}(2,3)(0,0)
\linethickness{.25pt}
\put(0,0){\line(1,0){2}}
\put(0,1){\line(1,0){2}}
\put(0,2){\line(1,0){2}}
\put(0,3){\line(1,0){2}}
\put(0,3){\line(0,-1){3}}
\put(1,3){\line(0,-1){3}}
\put(2,3){\line(0,-1){3}}
\put(.5,.5){\makebox(0,0){4}}
\put(1.5,.5){\makebox(0,0){5}}
\put(1.5,1.5){\makebox(0,0){6}}
\put(.5,1.5){\makebox(0,0){\begin{tiny}b\end{tiny}}}
\put(1.5,2.5){\makebox(0,0){\begin{tiny}a\end{tiny}}}
\end{picture}
$\stackrel{3\mapsto b}{\Longrightarrow}$
\begin{picture}(2,3)(0,0)
\linethickness{.25pt}
\put(0,0){\line(1,0){2}}
\put(0,1){\line(1,0){2}}
\put(0,2){\line(1,0){2}}
\put(0,3){\line(1,0){2}}
\put(0,3){\line(0,-1){3}}
\put(1,3){\line(0,-1){3}}
\put(2,3){\line(0,-1){3}}
\put(.5,.5){\makebox(0,0){4}}
\put(1.5,.5){\makebox(0,0){5}}
\put(.5,1.5){\makebox(0,0){3}}
\put(1.5,1.5){\makebox(0,0){6}}
\put(1.5,2.5){\makebox(0,0){\begin{tiny}a\end{tiny}}}
\end{picture}
$\stackrel{2\mapsto a}{\Longrightarrow}$
\begin{picture}(2,3)(0,0)
\linethickness{.25pt}
\put(0,0){\line(1,0){2}}
\put(0,1){\line(1,0){2}}
\put(0,2){\line(1,0){2}}
\put(0,3){\line(1,0){2}}
\put(0,3){\line(0,-1){3}}
\put(1,3){\line(0,-1){3}}
\put(2,3){\line(0,-1){3}}
\put(.5,.5){\makebox(0,0){4}}
\put(1.5,.5){\makebox(0,0){5}}
\put(.5,1.5){\makebox(0,0){3}}
\put(1.5,1.5){\makebox(0,0){6}}
\put(1.5,2.5){\makebox(0,0){2}}
\end{picture}
$\stackrel{1 \mbox{\begin{scriptsize} forced\end{scriptsize}}}{\Longrightarrow}$
\begin{picture}(2,3)(0,0)
\linethickness{.25pt}
\put(0,0){\line(1,0){2}}
\put(0,1){\line(1,0){2}}
\put(0,2){\line(1,0){2}}
\put(0,3){\line(1,0){2}}
\put(0,3){\line(0,-1){3}}
\put(1,3){\line(0,-1){3}}
\put(2,3){\line(0,-1){3}}
\put(.5,.5){\makebox(0,0){4}}
\put(1.5,.5){\makebox(0,0){5}}
\put(.5,1.5){\makebox(0,0){3}}
\put(1.5,1.5){\makebox(0,0){6}}
\put(.5,2.5){\makebox(0,0){1}}
\put(1.5,2.5){\makebox(0,0){2}}
\end{picture}~.\end{center}
\end{example}

\begin{remark}\label{rem:isom_vector_space_remark}
Since the composition $M^\mu \stackrel{\Phi}{\longrightarrow} \A \stackrel{\Psi}{\longrightarrow} M^\mu$ is the identity, it follows that $\A$ and $M^\mu$ are isomorphic as graded vector spaces.  This proof is a simple consequence of the fact that the monomials $\A$ coincide with the Garsia-Procesi basis $\B$.  We show this in the next subsection in Corollary~\ref{cor:A_and_M_isom_vect_spaces}.
\end{remark}

\subsection{\texorpdfstring{$\A$}{A(mu)} coincides with the Garsia-Procesi basis \texorpdfstring{$\B$}{B(mu)}}\label{subsec:A_equals_B_monomials}

Garsia and Procesi construct a tree~\cite[pg.87]{GP} that we call a \textit{GP-tree} to define their monomial basis $\B$.  In this subsection we modify this tree's construction to deliver more information.  For a given monomial $\X \in \B$, each path on the modified tree tells us how to construct a row-strict tableau $T$ such that $\Phi(T)$ equals $\X$.  In other words the paths on the tree give $\Psi$.  First we recall what Garsia and Procesi did.  Then we give an example that makes this algorithm more transparent.  Lastly we define our modification and give our specific results.

\begin{remark}
Although Garsia and Procesi's construction of a GP-tree mentions nothing of a dimension-ordering (recall Definition~\ref{def:dimension_ordering}), we find it clearer to explain the combinatorics of building their tree in Definition~\ref{def:tree_of_GP} using this concept.  They also use French-style Ferrers diagrams, but we will use the convention of having our tableaux flush top and left.
\end{remark}

\begin{definition}[GP-tree]\label{def:tree_of_GP}
If $\mu$ is a partition of $n$, then the \textit{GP-tree of $\mu$} is a tree with $n$ levels constructed as follows.  Let $\mu$ sit alone at the top Level $n$.  From a subdiagram $\mu^{(i)}$ at Level $i$, we branch down to exactly $r$ new subdiagrams at Level $i-1$ where $r$ equals the number of nonzero rows of $\mu^{(i)}$.  Note that this branching is \textit{injective}---that is, no two Level $i$ diagrams branch down to the same Level $i-1$ diagram.  Label these $r$ edges left to right with the labels $x_i^0, x_i^1, \ldots, x_i^{r-1}$.  Impose the dimension-ordering on $\mu^{(i)}$.  The subdiagram at the end of the edge labelled $x_i^j$ for some $j\in\{0,1,\ldots,r-1\}$ will be exactly $\mu^{(i)}$ with the box with dimension-ordering $j+1$ removed.  If a gap in a column is created by removing this box, then correct the gap by pushing up on this column to make a proper Young diagram instead of a composition.  At Level 1 there is a set of single box diagrams.  Instead of placing single boxes at this level, put the product of the edge labels from Level $n$ down to this vertex.  These monomials are the basis for $\B$~\cite[Theorem 3.1, pg.100]{GP}.
\end{definition}

\begin{example}[GP-tree for $\mu=(2,2)$]
Let $\mu=(2,2)$, which has shape\setlength{\unitlength}{.07in}
\begin{picture}(2,2)(0,0)
\linethickness{.25pt}
\put(0,0){\line(1,0){2}}
\put(0,1){\line(1,0){2}}
\put(0,2){\line(1,0){2}}
\put(0,2){\line(0,-1){2}}
\put(1,2){\line(0,-1){2}}
\put(2,2){\line(0,-1){2}}
\end{picture}~.  We start at the top Level 4 with the shape (2,2).  The first branching of the (2,2)-tree is
%\vspace{-.25in}
$$
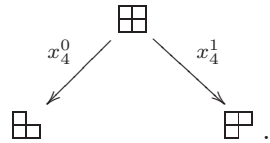
\xymatrix{
& \begin{picture}(2,2)(0,0)
\linethickness{.25pt}
\put(0,0){\line(1,0){2}}
\put(0,1){\line(1,0){2}}
\put(0,2){\line(1,0){2}}
\put(0,2){\line(0,-1){2}}
\put(1,2){\line(0,-1){2}}
\put(2,2){\line(0,-1){2}}
\end{picture} \ar[dl]_{x_4^0} \ar[dr]^{x_4^1}\\
\begin{picture}(2,2)(0,0)
\linethickness{.25pt}
\put(0,0){\line(1,0){2}}
\put(0,1){\line(1,0){2}}
\put(0,2){\line(1,0){1}}
\put(0,2){\line(0,-1){2}}
\put(1,2){\line(0,-1){2}}
\put(2,1){\line(0,-1){1}}
\end{picture} & & \begin{picture}(2,2)(0,0)
\linethickness{.25pt}
\put(0,0){\line(1,0){1}}
\put(0,1){\line(1,0){2}}
\put(0,2){\line(1,0){2}}
\put(0,2){\line(0,-1){2}}
\put(1,2){\line(0,-1){2}}
\put(2,2){\line(0,-1){1}}
\end{picture}~.
}
$$
But we make the bottom-left non-standard diagram into a proper Young diagram by pushing the bottom-right box up the column.  In the Figure~\ref{fig:GP_Tree_Example}, we show the completed GP-tree.  Observe that the six monomials at Level 1 are the Garsia-Procesi basis $\B$.
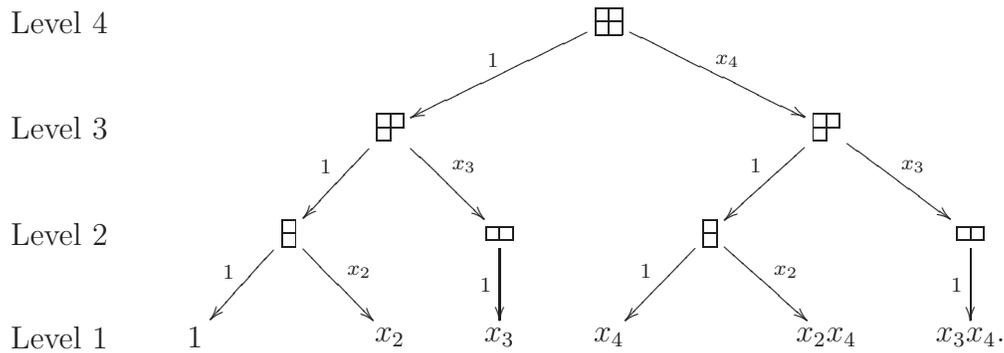
\begin{figure}[!ht]
$$
\xymatrix{
\mbox{Level 4} & & & & & \begin{picture}(2,2)(0,0)
\linethickness{.25pt}
\put(0,0){\line(1,0){2}}
\put(0,1){\line(1,0){2}}
\put(0,2){\line(1,0){2}}
\put(0,2){\line(0,-1){2}}
\put(1,2){\line(0,-1){2}}
\put(2,2){\line(0,-1){2}}
\end{picture} \ar[dll]_1 \ar[drr]^{x_4}\\
\mbox{Level 3} & & & \begin{picture}(2,2)(0,0)
\linethickness{.25pt}
\put(0,0){\line(1,0){1}}
\put(0,1){\line(1,0){2}}
\put(0,2){\line(1,0){2}}
\put(0,2){\line(0,-1){2}}
\put(1,2){\line(0,-1){2}}
\put(2,2){\line(0,-1){1}}
\end{picture} \ar[dl]_1 \ar[dr]^{x_3} & & & & \begin{picture}(2,2)(0,0)
\linethickness{.25pt}
\put(0,0){\line(1,0){1}}
\put(0,1){\line(1,0){2}}
\put(0,2){\line(1,0){2}}
\put(0,2){\line(0,-1){2}}
\put(1,2){\line(0,-1){2}}
\put(2,2){\line(0,-1){1}}
\end{picture} \ar[dl]_{1} \ar[dr]^{x_3}\\
\mbox{Level 2} & & \begin{picture}(1,2)(0,0)
\linethickness{.25pt}
\put(0,0){\line(1,0){1}}
\put(0,1){\line(1,0){1}}
\put(0,2){\line(1,0){1}}
\put(0,2){\line(0,-1){2}}
\put(1,2){\line(0,-1){2}}
\end{picture} \ar[dl]_1\ar[dr]^{x_2} & & \begin{picture}(2,1)(0,0)
\linethickness{.25pt}
\put(0,0){\line(1,0){2}}
\put(0,1){\line(1,0){2}}
\put(0,1){\line(0,-1){1}}
\put(1,1){\line(0,-1){1}}
\put(2,1){\line(0,-1){1}}
\end{picture} \ar[d]_{1} & & \begin{picture}(1,2)(0,0)
\linethickness{.25pt}
\put(0,0){\line(1,0){1}}
\put(0,1){\line(1,0){1}}
\put(0,2){\line(1,0){1}}
\put(0,2){\line(0,-1){2}}
\put(1,2){\line(0,-1){2}}
\end{picture} \ar[dl]_{1}\ar[dr]^{x_2} & & \begin{picture}(2,1)(0,0)
\linethickness{.25pt}
\put(0,0){\line(1,0){2}}
\put(0,1){\line(1,0){2}}
\put(0,1){\line(0,-1){1}}
\put(1,1){\line(0,-1){1}}
\put(2,1){\line(0,-1){1}}
\end{picture} \ar[d]_{1}\\
\mbox{Level 1} & 1 & & x_2 & x_3 & x_4 & & x_2x_4 & x_3x_4.
}$$
\caption{\label{fig:GP_Tree_Example}The GP-tree for $\mu=(2,2)$.}
\end{figure}
\end{example}
\vspace{-.1in}

\begin{remark}
Each time a subdiagram is altered to make it look like a proper Young diagram, we lose information that can be used to reconstruct a row-strict tableau in $M^{\mu}$ from a given monomial in $\B$.  The construction below will take this into account, and give the precise prescription for constructing a filling from a monomial in $\B$.
\end{remark}

\begin{definition}[Modified GP-tree]\label{def:Modified_GP_Tree}
Let $\mu$ be a partition of $n$.  The \textit{modified GP-tree for $\mu$} is a tree with $n+2$ levels.  The top is Level $n$ with diagram $\mu$ at its vertex.  The branching and edge labelling rules are the same as in the GP-tree.  The crucial modification from the GP-tree is the diagram at the end of a branching edge.
\begin{itemize}
 \item When branching down from Level $i$ down to Level $i-1$ for $i\geq1$, the new diagram at Level $i-1$ will be a composition $\mu^{(i-1)}$ of $i-1$ with a partial filling of the values $i,\ldots,n$ in the remaining $n-(i-1)$ boxes of $\mu$.  In the diagram at the end of the edge labelled $x_i^j$, instead of removing the box with dimension-ordering $j+1$ place the value $i$ in this box.
\end{itemize}
Place the label 1 on the edge from Level 0 down to its unique corresponding leaf at the bottom $(n+2)^{th}$ level, which we call Level B.  Label each leaf at Level B with the product of the edge labels on the path connecting the root vertex of the tree with this leaf.
\end{definition}

\begin{remark}
Observe that we never move a box as was done in the GP-tree to create a Young diagram from a composition.  There are now two sublevels below Level 1: Level 0 has a filling of $\mu$ constructed through this tree, and Level B has the monomials in $\B$ coming from the product of the edge labels on the paths.  Theorem~\ref{thm:Getting_(h,mu)-fillings_from_GP-tree} highlights a profound relationship between these two levels.
\end{remark}

\begin{example}
Again consider the shape $\mu=(2,2)$.  Dimension order the Level 4 diagram to get \setlength{\unitlength}{.1in}\begin{picture}(2,2)(0,0)
\linethickness{.25pt}
\put(0,0){\line(1,0){2}}\put(0,1){\line(1,0){2}}\put(0,2){\line(1,0){2}}
\put(0,2){\line(0,-1){2}}\put(1,2){\line(0,-1){2}}\put(2,2){\line(0,-1){2}}
\put(1.2,.1){\begin{scriptsize}$b$\end{scriptsize}}
\put(1.2,1.1){\begin{scriptsize}$a$\end{scriptsize}}
\end{picture}.  Branch downward left placing 4 in the dimension-ordered box we labelled $a$.  Branch downward right placing 4 in the dimension-ordered box we labelled $b$.  Ignoring the filled box, impose dimension-orderings on both compositions: on the left, of shape $(1,2)$; and on the right, of shape $(2,1)$.  This gives:

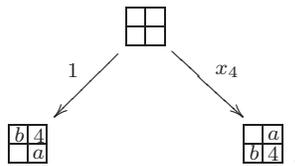
\setlength{\unitlength}{.1in}
$$\xymatrix{
&\begin{picture}(2,2)(0,0)
\linethickness{.25pt}
\put(0,0){\line(1,0){2}}\put(0,1){\line(1,0){2}}\put(0,2){\line(1,0){2}}
\put(0,2){\line(0,-1){2}}\put(1,2){\line(0,-1){2}}\put(2,2){\line(0,-1){2}}\end{picture} \ar[dl]_1 \ar[dr]^{x_4} \\
\begin{picture}(2,2)(0,0)
\linethickness{.25pt}
\put(0,0){\line(1,0){2}}\put(0,1){\line(1,0){2}}\put(0,2){\line(1,0){2}}
\put(0,2){\line(0,-1){2}}\put(1,2){\line(0,-1){2}}\put(2,2){\line(0,-1){2}}
\put(1.3,1.05){\begin{scriptsize}4\end{scriptsize}} \put(1.25,.2){\begin{scriptsize}$a$\end{scriptsize}}
\put(.3,1.1){\begin{scriptsize}$b$\end{scriptsize}}\end{picture} & & \begin{picture}(2,2)(0,0)
\linethickness{.25pt}
\put(0,0){\line(1,0){2}}\put(0,1){\line(1,0){2}}\put(0,2){\line(1,0){2}}
\put(0,2){\line(0,-1){2}}\put(1,2){\line(0,-1){2}}\put(2,2){\line(0,-1){2}}
\put(1.25,.1){\begin{scriptsize}4\end{scriptsize}}\put(1.25,1.2){\begin{scriptsize}$a$\end{scriptsize}}
\put(.3,.15){\begin{scriptsize}$b$\end{scriptsize}}\end{picture}~.
}$$
\medskip

In each of these subdiagrams branch down to the next Level by placing 3 in the appropriate dimension-ordered boxes.  The completed tree is given in Figure~\ref{fig:Modified_GP_Tree_Example}.

\begin{figure}[!ht]
\setlength{\unitlength}{.1in}
$$
\xymatrix{
\mbox{Level 4} & & & & & \begin{picture}(2,2)(0,0)
\linethickness{.25pt}
\put(0,0){\line(1,0){2}}
\put(0,1){\line(1,0){2}}
\put(0,2){\line(1,0){2}}
\put(0,2){\line(0,-1){2}}
\put(1,2){\line(0,-1){2}}
\put(2,2){\line(0,-1){2}}
\end{picture} \ar[dll]_1 \ar[drr]^{x_4}\\
%% LEVEL 3
\mbox{Level 3} & & & \begin{picture}(2,2)(0,0)
\linethickness{.25pt}
\put(0,0){\line(1,0){2}}\put(0,1){\line(1,0){2}}\put(0,2){\line(1,0){2}}
\put(0,2){\line(0,-1){2}}\put(1,2){\line(0,-1){2}}\put(2,2){\line(0,-1){2}}
\put(1.25,1.05){\begin{scriptsize}4\end{scriptsize}} \end{picture} \ar[dl]_1 \ar[dr]^{x_3} & & & & \begin{picture}(2,2)(0,0)
\linethickness{.25pt}
\put(0,0){\line(1,0){2}}\put(0,1){\line(1,0){2}}\put(0,2){\line(1,0){2}}
\put(0,2){\line(0,-1){2}}\put(1,2){\line(0,-1){2}}\put(2,2){\line(0,-1){2}}
\put(1.25,.05){\begin{scriptsize}4\end{scriptsize}} \end{picture} \ar[dl]_{1} \ar[dr]^{x_3}\\
%% LEVEL 2
\mbox{Level 2} & & \begin{picture}(2,2)(0,0)
\linethickness{.25pt}
\put(0,0){\line(1,0){2}}\put(0,1){\line(1,0){2}}\put(0,2){\line(1,0){2}}
\put(0,2){\line(0,-1){2}}\put(1,2){\line(0,-1){2}}\put(2,2){\line(0,-1){2}}
\put(1.25,1.05){\begin{scriptsize}4\end{scriptsize}} \put(1.25,.05){\begin{scriptsize}3\end{scriptsize}}\end{picture} \ar[dl]_1\ar[dr]^{x_2} & & \begin{picture}(2,2)(0,0)
\linethickness{.25pt}
\put(0,0){\line(1,0){2}}\put(0,1){\line(1,0){2}}\put(0,2){\line(1,0){2}}
\put(0,2){\line(0,-1){2}}\put(1,2){\line(0,-1){2}}\put(2,2){\line(0,-1){2}}
\put(1.25,1.05){\begin{scriptsize}4\end{scriptsize}} \put(.25,1.05){\begin{scriptsize}3\end{scriptsize}}\end{picture} \ar[d]_{1} & & \begin{picture}(2,2)(0,0)
\linethickness{.25pt}
\put(0,0){\line(1,0){2}}\put(0,1){\line(1,0){2}}\put(0,2){\line(1,0){2}}
\put(0,2){\line(0,-1){2}}\put(1,2){\line(0,-1){2}}\put(2,2){\line(0,-1){2}}
\put(1.25,.05){\begin{scriptsize}4\end{scriptsize}} \put(1.25,1.05){\begin{scriptsize}3\end{scriptsize}}\end{picture} \ar[dl]_{1}\ar[dr]^{x_2} & & \begin{picture}(2,2)(0,0)
\linethickness{.25pt}
\put(0,0){\line(1,0){2}}\put(0,1){\line(1,0){2}}\put(0,2){\line(1,0){2}}
\put(0,2){\line(0,-1){2}}\put(1,2){\line(0,-1){2}}\put(2,2){\line(0,-1){2}}
\put(1.25,.05){\begin{scriptsize}4\end{scriptsize}} \put(.25,.05){\begin{scriptsize}3\end{scriptsize}}\end{picture} \ar[d]_{1}\\
%% LEVEL 1
\mbox{Level 1} & \begin{picture}(2,2)(0,0)
\linethickness{.25pt}
\put(0,0){\line(1,0){2}}\put(0,1){\line(1,0){2}}\put(0,2){\line(1,0){2}}
\put(0,2){\line(0,-1){2}}\put(1,2){\line(0,-1){2}}\put(2,2){\line(0,-1){2}}
\put(1.25,1.05){\begin{scriptsize}4\end{scriptsize}} \put(1.25,.05){\begin{scriptsize}3\end{scriptsize}} \put(.25,1.05){\begin{scriptsize}2\end{scriptsize}}\end{picture}\ar[d]_1 & & \begin{picture}(2,2)(0,0)
\linethickness{.25pt}
\put(0,0){\line(1,0){2}}\put(0,1){\line(1,0){2}}\put(0,2){\line(1,0){2}}
\put(0,2){\line(0,-1){2}}\put(1,2){\line(0,-1){2}}\put(2,2){\line(0,-1){2}}
\put(1.25,1.05){\begin{scriptsize}4\end{scriptsize}} \put(1.25,.05){\begin{scriptsize}3\end{scriptsize}} \put(.25,.05){\begin{scriptsize}2\end{scriptsize}}\end{picture}\ar[d]_1 & \begin{picture}(2,2)(0,0)
\linethickness{.25pt}
\put(0,0){\line(1,0){2}}\put(0,1){\line(1,0){2}}\put(0,2){\line(1,0){2}}
\put(0,2){\line(0,-1){2}}\put(1,2){\line(0,-1){2}}\put(2,2){\line(0,-1){2}}
\put(1.25,1.05){\begin{scriptsize}4\end{scriptsize}} \put(.25,1.05){\begin{scriptsize}3\end{scriptsize}} \put(1.25,.05){\begin{scriptsize}2\end{scriptsize}}\end{picture}\ar[d]_1 & \begin{picture}(2,2)(0,0)
\linethickness{.25pt}
\put(0,0){\line(1,0){2}}\put(0,1){\line(1,0){2}}\put(0,2){\line(1,0){2}}
\put(0,2){\line(0,-1){2}}\put(1,2){\line(0,-1){2}}\put(2,2){\line(0,-1){2}}
\put(1.25,1.05){\begin{scriptsize}3\end{scriptsize}} \put(1.25,.05){\begin{scriptsize}4\end{scriptsize}} \put(.25,1.05){\begin{scriptsize}2\end{scriptsize}}\end{picture}\ar[d]_1 & & \begin{picture}(2,2)(0,0)
\linethickness{.25pt}
\put(0,0){\line(1,0){2}}\put(0,1){\line(1,0){2}}\put(0,2){\line(1,0){2}}
\put(0,2){\line(0,-1){2}}\put(1,2){\line(0,-1){2}}\put(2,2){\line(0,-1){2}}
\put(1.25,1.05){\begin{scriptsize}3\end{scriptsize}} \put(1.25,.05){\begin{scriptsize}4\end{scriptsize}} \put(.25,.05){\begin{scriptsize}2\end{scriptsize}}\end{picture}\ar[d]_1 & \begin{picture}(2,2)(0,0)
\linethickness{.25pt}
\put(0,0){\line(1,0){2}}\put(0,1){\line(1,0){2}}\put(0,2){\line(1,0){2}}
\put(0,2){\line(0,-1){2}}\put(1,2){\line(0,-1){2}}\put(2,2){\line(0,-1){2}}
\put(.25,.05){\begin{scriptsize}3\end{scriptsize}} \put(1.25,.05){\begin{scriptsize}4\end{scriptsize}} \put(1.25,1.05){\begin{scriptsize}2\end{scriptsize}}\end{picture}\ar[d]_1\\
%% LEVEL 0
\mbox{Level 0} & \begin{picture}(2,2)(0,0)
\linethickness{.25pt}
\put(0,0){\line(1,0){2}}\put(0,1){\line(1,0){2}}\put(0,2){\line(1,0){2}}
\put(0,2){\line(0,-1){2}}\put(1,2){\line(0,-1){2}}\put(2,2){\line(0,-1){2}}
\put(1.25,1.05){\begin{scriptsize}4\end{scriptsize}} \put(1.25,.05){\begin{scriptsize}3\end{scriptsize}} \put(.25,1.05){\begin{scriptsize}2\end{scriptsize}} \put(.25,.05){\begin{scriptsize}1\end{scriptsize}}\end{picture}\ar[d]_1 & & \begin{picture}(2,2)(0,0)
\linethickness{.25pt}
\put(0,0){\line(1,0){2}}\put(0,1){\line(1,0){2}}\put(0,2){\line(1,0){2}}
\put(0,2){\line(0,-1){2}}\put(1,2){\line(0,-1){2}}\put(2,2){\line(0,-1){2}}
\put(1.25,1.05){\begin{scriptsize}4\end{scriptsize}} \put(1.25,.05){\begin{scriptsize}3\end{scriptsize}} \put(.25,.05){\begin{scriptsize}2\end{scriptsize}} \put(.25,1.05){\begin{scriptsize}1\end{scriptsize}}\end{picture}\ar[d]_1 & \begin{picture}(2,2)(0,0)
\linethickness{.25pt}
\put(0,0){\line(1,0){2}}\put(0,1){\line(1,0){2}}\put(0,2){\line(1,0){2}}
\put(0,2){\line(0,-1){2}}\put(1,2){\line(0,-1){2}}\put(2,2){\line(0,-1){2}}
\put(1.25,1.05){\begin{scriptsize}4\end{scriptsize}} \put(.25,1.05){\begin{scriptsize}3\end{scriptsize}} \put(1.25,.05){\begin{scriptsize}2\end{scriptsize}} \put(.25,.05){\begin{scriptsize}1\end{scriptsize}}\end{picture}\ar[d]_1 & \begin{picture}(2,2)(0,0)
\linethickness{.25pt}
\put(0,0){\line(1,0){2}}\put(0,1){\line(1,0){2}}\put(0,2){\line(1,0){2}}
\put(0,2){\line(0,-1){2}}\put(1,2){\line(0,-1){2}}\put(2,2){\line(0,-1){2}}
\put(1.25,1.05){\begin{scriptsize}3\end{scriptsize}} \put(1.25,.05){\begin{scriptsize}4\end{scriptsize}} \put(.25,1.05){\begin{scriptsize}2\end{scriptsize}} \put(.25,.05){\begin{scriptsize}1\end{scriptsize}}\end{picture}\ar[d]_1 & & \begin{picture}(2,2)(0,0)
\linethickness{.25pt}
\put(0,0){\line(1,0){2}}\put(0,1){\line(1,0){2}}\put(0,2){\line(1,0){2}}
\put(0,2){\line(0,-1){2}}\put(1,2){\line(0,-1){2}}\put(2,2){\line(0,-1){2}}
\put(1.25,1.05){\begin{scriptsize}3\end{scriptsize}} \put(1.25,.05){\begin{scriptsize}4\end{scriptsize}} \put(.25,.05){\begin{scriptsize}2\end{scriptsize}} \put(.25,1.05){\begin{scriptsize}1\end{scriptsize}}\end{picture}\ar[d]_1 & \begin{picture}(2,2)(0,0)
\linethickness{.25pt}
\put(0,0){\line(1,0){2}}\put(0,1){\line(1,0){2}}\put(0,2){\line(1,0){2}}
\put(0,2){\line(0,-1){2}}\put(1,2){\line(0,-1){2}}\put(2,2){\line(0,-1){2}}
\put(.25,.05){\begin{scriptsize}3\end{scriptsize}} \put(1.25,.05){\begin{scriptsize}4\end{scriptsize}} \put(1.25,1.05){\begin{scriptsize}2\end{scriptsize}} \put(.25,1.05){\begin{scriptsize}1\end{scriptsize}}\end{picture}\ar[d]_1\\
%% LEVEL B
\mbox{Level B} & 1 & & x_2 & x_3 & x_4 & & x_2x_4 & x_3x_4.
}$$
\vspace{-.1in}
\caption{\label{fig:Modified_GP_Tree_Example}The modified GP-tree for $\mu=(2,2)$.}
\end{figure}
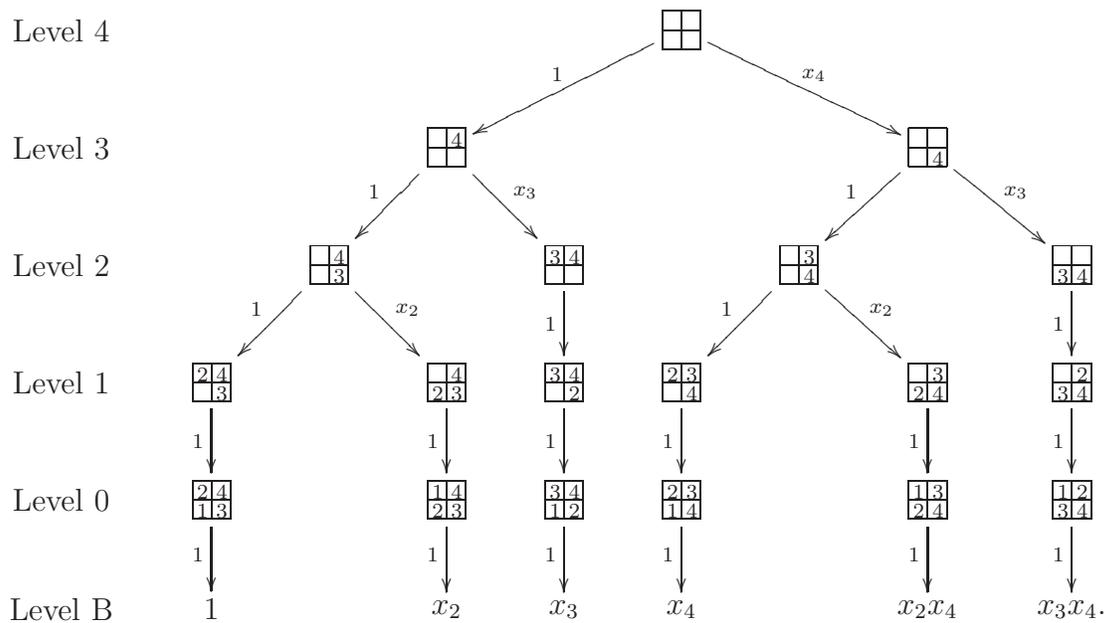
\end{example}

\begin{theorem}\label{thm:Getting_(h,mu)-fillings_from_GP-tree}
Let $\mu$ be a Young diagram and consider its corresponding modified GP-tree.  Each of the fillings at Level 0 are $\JT$-fillings.  Moreover given a filling $T$, the image $\Phi(T)$ will be the monomial $\X \in \B$ that is the neighbor of $T$ in Level B.
\end{theorem}

\begin{proof}
Fix a partition $\mu$ of $n$.  Consider a path in the modified GP-tree for $\mu$.  From Level $n-1$ to Level 0, the numbers $n$ through 1 are placed in reverse-order in the dimension-ordered boxes.  Finally at Level 0, a filling $T$ satisfying the Subfilling Property is completed.  By Lemma~\ref{lem:dim_ordered_filling_implies_row_strict}, $T$ is row-strict and hence is an $\JT$-filling.

Let $T$ be a tableau at Level 0, and let $\X = x_2^{\alpha_2}\cdots x_n^{\alpha_n}$ at Level B be the monomial below $T$.  We claim that $\Phi(T)=\X$.  By Lemma~\ref{lem:dim_order_filling_gives_these_many_dimpairs}, for each $i$ the cardinality of $\DP_i^{T^{(i)}}$ equals $\alpha_i$ where $T^{(i)}$ is the $i^{th}$-subfilling of $T$ (recall Definition~\ref{def:subfilling}).  By Lemma~\ref{lem:CompositionDP_equals_TabloidDP}, the value $|\DP_i^{T^{(i)}}|$ will equal $|\DP_i^T|$.  Hence $\DP_i^T$ has exactly $\alpha_i$ dimension pairs so $\Phi(T)=\X$ as desired.
\end{proof}

A surprising application of the modified GP-tree is to count the elements of $M^{\mu}$.  The corollary gives $\A = \B$.

\begin{theorem}\label{thm:cool_modified_tree_application}
Let $\mu = (\mu_1,\ldots,\mu_k)$ be a partition of $n$.  The number of paths in the modified GP-tree for $\mu$ is exactly $\dfrac{n!}{\mu_1!\cdots\mu_k!}$.  In particular, Level 0 is composed of exactly all possible row-strict tableaux of shape $\mu$.
\end{theorem}

\begin{proof}
Firstly, the number of paths in the modified GP-tree is the same as in the standard GP-tree.  Garsia and Procesi prove~\cite[Prop. 3.2]{GP} that the dimension of their quotient ring presentation equals $\frac{n!}{\mu_1!\cdots\mu_k!}$.  A direct counting argument proves $|M^{\mu}|$ equals this same number since row-strict tableaux correspond bijectively to the collection of $k$ subsets of $\{1,2,\ldots,n\}$ with $\mu_1,\ldots,\mu_k$ elements, respectively.  Hence $|\B|$ equals $|M^{\mu}|$.  Each of the paths in the modified GP-tree gives a unique $\JT$-filling at Level 0 by construction, and there are $|\B|$ such paths.  Recall that the $\JT$-fillings in this case are the row-strict fillings.  Thus the fillings at Level 0 are not just a subset of row-strict tableaux.  Level 0 is exactly $M^{\mu}$.
\end{proof}

\begin{corollary}\label{cor:A_monom_equal_B_monom_Springer_Case}
The sets of monomials $\A$ and $\B$ coincide.
\end{corollary}

\begin{proof}
This follows since the image of all $\JT$-fillings under $\Phi$ is $\A$.  The image of the Level 0 fillings in the modified GP-tree is $\B$.  Theorem~\ref{thm:cool_modified_tree_application} implies that both the set of $\JT$-fillings and the Level 0 fillings coincide, and hence it follows that $\A = \B$.
\end{proof}

\begin{corollary}\label{cor:A_and_M_isom_vect_spaces}
$\A$ and $M^{\mu}$ are isomorphic as graded vector spaces.
\end{corollary}

\begin{proof}
By Theorem~\ref{thm:GP_to_Tab}, the composition $\Phi\circ\Psi$ is the identity on $\A$.  Since $\A$ equals $\B$, Theorem~\ref{thm:cool_modified_tree_application} implies the cardinality of $\A$ equals the cardinality of the generating set of row-strict tableaux in $M^{\mu}$.  Also, $\Phi$ is a degree-preserving map so $\A$ and $M^{\mu}$ are isomorphic as graded vector spaces.
\end{proof}

%%%%%%%%%%%%%%%%%%%%%%%%%%%%%%%%%%%%%%%%%%%%%%%%%%%%%%
%%%%%%%%%%%%%%%%%%%%%%%%%%%%%%%%%%%%%%%%%%%%%%%%%%%%%%
%%%%%%%%%%%%%%%%%%%%%%%%%%%%%%%%%%%%%%%%%%%%%%%%%%%%%%
%%%%%%%%%%%%%%%%%%%%%%%%%%%%%%%%%%%%%%%%%%%%%%%%%%%%%%
%%%%%%%%%%%%%%%%%%%%%%%%%%%%%%%%%%%%%%%%%%%%%%%%%%%%%%
%%%%%%%%%%%%%%%%%%%%%%%%%%%%%%%%%%%%%%%%%%%%%%%%%%%%%%
%%%%%%%%%%%%%%%%%%%%%%%%%%%%%%%%%%%%%%%%%%%%%%%%%%%%%%
%%%%%%%%%%%%%%%%%%%%%%%%%%%%%%%%%%%%%%%%%%%%%%%%%%%%%%
%%%%%%%%%%%%%%%%%% S E C T I O N  3 %%%%%%%%%%%%%%%%%%

\section{The regular nilpotent Hessenberg setting}\label{sec:RegNilpHess_Setting}
When we fix the Hessenberg function $h=(1,2,\ldots,n)$ and let the shape $\mu$ (equivalently, the nilpotent $X$) vary, the image of $\Phi$ is a very meaningful set of monomials: the Garsia-Procesi basis $\B$ for the cohomology ring of the Springer variety, $H^*(\mathfrak{S}_X)$.  Moreover there is a well-defined inverse map $\Psi$.  What if we now let $h$ vary?  Are these new monomials in the image of $\Phi$ still meaningful?  For other Hessenberg functions, the map $\Psi$ no longer maps reliably back to the original filling.  For example if $h=(1,3,3)$ then $\Phi\left(\TwoOneTab[.5]321\right) = x_3$, but $\Psi(x_3) = \TwoOneTab[.5]123$.  Attempts so far to define an inverse map that work for all Hessenberg functions and all shapes $\mu$ have been unsuccessful.

However, when we fix the shape $\mu=(n)=\setlength{\unitlength}{.15in}\begin{picture}(5,1)(0,0) \linethickness{.25pt}\multiput(0,0)(0,1){2}{\line(1,0){5}}
\multiput(0,0)(1,0){3}{\line(0,1){1}}
\put(2.35,.15){$\cdots$}\put(4,0){\line(0,1){1}}\put(5,0){\line(0,1){1}}\end{picture}$ (equivalently, fix the nilpotent $X$ to have exactly one Jordan block) and let the functions $h$ vary, we get an important family of varieties called the \textit{regular nilpotent Hessenberg varieties}.  In this setting the image of $\Phi$ is indeed a meaningful set of monomials $\Ah$.  They coincide with a basis $\Bh$ of a polynomial quotient ring $R/J_h$ which we conjecture (with supporting data) is a presentation for the cohomology ring of the regular nilpotent Hessenberg varieties.  In this section we will fill in the details of Figure~\ref{fig:The_Reg_Nilp_Triangle}.
\begin{figure}[!ht]
$$
\xymatrix{
& H^*(\mathfrak{H}(X,h)) \ar@{<~>}[dl] \ar@{<-->}[dr]^{\cong\;?}\\
\txt{$\JT$-fillings\\spanning $M^{h,\mu}$}\ar@{<->}[rr]^{\stackrel{\Phi}{\longrightarrow}}_{\stackrel{\longleftarrow}{\Psi_h}} & & \txt{$R/J_h$ with\\$\Ah=\Bh$\\basis}
}
$$
%\vspace{-.25in}
\caption{\label{fig:The_Reg_Nilp_Triangle}Regular nilpotent Hessenberg setting.}
\end{figure}
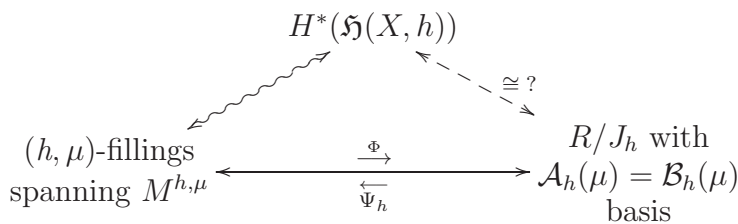

Recall that the dimensions of the graded parts of $H^*(\mathfrak{H}(X,h))$ are combinatorially described by the $\JT$-fillings.  This gives the geometric description of the cohomology ring denoted by the left edge of the triangle.  The formal $\mathbb{Q}$-linear span of the $\JT$-fillings is denoted $M^{h,\mu}$.  The map $\Phi$ is a graded vector space morphism from $M^{h,\mu}$ to the span of monomials $\Ah$.  In Subsection~\ref{subsec:Ah_equals_Bh_RegNilpCase}, we show that $\Phi$ is actually a graded isomorphism, completing the bottom leg of the triangle.  In Theorem~\ref{thm:algebraic_view_of_R/J_h}, we conclude that the generators of degree $i$ in $R/J_h$ correspond to $\JT$-fillings of dimension $i$ and hence to the $2i^{th}$ Betti number of the regular nilpotent Hessenberg varieties.  This gives a view towards an algebraic description of $H^*(\mathfrak{H}(X,h))$.

In Subsection~\ref{subsec:building_h_tab_tree}, for a given Hessenberg function $h$ we build an $h$-tableau tree.  This tree assumes the role that the modified GP-tree filled in Subsection~\ref{subsec:A_equals_B_monomials}.  In Subsection~\ref{subsec:Psi_h_map}, we construct the inverse map $\Psi_h$ from the span of the monomials $\Ah$ to the the vector space $M^{h,\mu}$.  Finally in Subsection~\ref{subsec:Ah_equals_Bh_RegNilpCase}, we show that the monomials $\Ah$ coincide with the basis $\Bh$ of the quotient $R/J_h$.

%%%%%%%%%%%%%%%%%%%%%%%%%%%%%%%%%%%%%%%%%%%%%%%%
%%%%%%%%%%%% The MOVED subsection %%%%%%%%%%%%%%

\subsection{The ideal \texorpdfstring{$J_h$}{Jh}, the quotient ring \texorpdfstring{$R/J_h$}{R/Jh}, and its basis \texorpdfstring{$\Bh$}{Bh}}\label{subsec:forthcomingwork}
We briefly describe the construction of the quotient ring $R/J_h$ where $R$ is the polynomial ring $\mathbb{Z}[x_1,\ldots,x_n]$ and $J_h$ is a combinatorially-described ideal generated by partial symmetric functions.  In Section~\ref{sec:evidence}, we offer evidence leading us to believe that this quotient may indeed be a presentation for the cohomology ring $H^*(\mathfrak{H}(X,h))$ of all regular nilpotent Hessenberg varieties.  The strengths of this presentation are in its ease of construction and the manner in which it reveals aspects of the integral cohomology of these special varieties.  In forthcoming work~\cite{Mb2} we explore the rich development of the ideal $J_h$ and its intimate connection to a generalization of the Tanisaki ideal (see Subsection~\ref{subsec:forthcoming_work}).

\begin{definition}  Let $S\subseteq\{x_1,\ldots,x_n\}$.  We define $\tilde{e}_r(S)$ to be the \textit{modified complete symmetric function} of degree $r$ in the variables $S$.  For example, $\tilde{e}_2(x_3,x_4) = x_3^2 + x_3x_4 + x_4^2$.
\end{definition}

\begin{definition}[Degree tuple]\label{def:beta_tuple}
Let $h=(h_1,h_2,\ldots,h_n)$ be a Hessenberg function.  The \textit{degree tuple} corresponding to $h$ is $\beta=(\beta_n,\beta_{n-1},\ldots,\beta_1)$ where $\beta_i = i - \#\{h_k \;|\; h_k < i\}$ for each $1 \leq i \leq n$.
\end{definition}

\begin{remark}
We call it a degree tuple because its entries are the degrees of the generating functions for the ideal $J_h$.  The convention of listing the $\beta_i$ in descending subscript order in a degree tuple highlights that the $i^{th}$ entry of a tuple corresponds to a generating function in exactly $i$ variables.  Degree tuples have many rich connections to Hessenberg functions, Dyck paths, Catalan numbers, and other combinatorial data.  These connections are explored more in the author's thesis~\cite{Mb-thesis}.
\end{remark}

\begin{definition}[The ideal $J_h$]
Let $h=(h_1,\dots,h_n)$ be a Hessenberg function with corresponding degree tuple $\beta=(\beta_n,\beta_{n-1},\ldots,\beta_1)$.  The ideal $J_h$ is defined as follows:
$$J_h:=\langle \tilde{e}_{\beta_n}(x_n),\tilde{e}_{\beta_{n-1}}(x_{n-1},x_n),\ldots,\tilde{e}_{\beta_1}(x_1,\ldots,x_n)\rangle.$$
\end{definition}

Proof of the following theorem involves commutative algebra that is beyond the scope of this paper.  Details can be found in the author's thesis~\cite{Mb-thesis}.

\begin{theorem}[A Basis for $R/J_h$]\label{thm:Basis_for_R/J}  Let $J_h$ be the ideal corresponding to the Hessenberg function $h$.  Then $R/J_h$ has the basis
$$\Bh := \left\lbrace x_1^{\alpha_1} x_2^{\alpha_2} \cdots x_n^{\alpha_n} \; \vline \; 0\leq\alpha_i\leq\beta_i-1, i=1,\ldots,n \right\rbrace.$$
\end{theorem}

%%%%%%%%%%%%%%%%%%%%%%%%%%%%%%%%%%%%%%%%%%%%%%%
%%%%%%%%%%%%%%%%%%%%%%%%%%%%%%%%%%%%%%%%%%%%%%%

\subsection{Constructing an \texorpdfstring{$h$}{h}-tableau-tree}\label{subsec:building_h_tab_tree}
Analogous to the Springer case we first build a tree, which we call an \textit{$h$-tree}, whose leaves give a basis $\Bh$ for the quotient ring $R/J_h$.  As with the modified GP-tree, we then take these leaves and describe how to construct a corresponding $\JT$-filling.  We label the vertices of the $h$-tree to produce a graph which we call an \textit{$h$-tableau-tree}.

\begin{remark}
In the Springer setting of Section~\ref{sec:Springer_Setting}, the levels in the trees are labelled in descending order from the top Level $n$ down to Level 1 in the case of the GP-tree (with the additional lower Levels 0 and B in the case of the modified GP-tree).  This descending label convention is meant to reflect the method of how to build the $\JT$-fillings in this Springer setting by inserting the numbers 1 thru $n$ in descending order.  However in the regular nilpotent setting of the current section, we build the $\JT$-fillings by inserting the numbers 1 thru $n$ in \emph{ascending} order.  The trees in this section reflect this method by being labelled from the top Level 1 down to Level $n+1$.
\end{remark}

\begin{definition}[$h$-tree]
Given a Hessenberg function $h=(h_1,h_2,\ldots,h_n)$, the corresponding \textit{$h$-tree} has $n+1$ levels labelled from the top Level 1 to the bottom Level $n+1$.  We start with one vertex at Level~1.  For $i\in\{2,\ldots,n\}$, each vertex at Level $i-1$ has exactly $\beta_i$ edges directed down to Level $i$ injectively (that is, no two Level $i-1$ vertices share an edge with the same Level $i$ vertex).  For each of the vertices at Level $i-1$, label the $\beta_i$ edges directed down to Level $i$ with the labels $\{x_i^{\beta_i-1}, x_i^{\beta_i-2}, \ldots, x_i^2, x_i, 1\}$ from left to right.  Connect each vertex at Level $n$ to a unique leaf at Level $n+1$, and label the corresponding edges with the value 1.  Label each leaf at Level $n+1$ with  the product of the edge labels of the path connecting the root vertex on Level 1 with this leaf.
\end{definition}

We omit the proof of the following proposition for it is a direct consequence of the definition of the basis $\Bh$ given in Theorem~\ref{thm:Basis_for_R/J} and the construction of an $h$-tree.

\begin{proposition}\label{thm:leaves_remarks_and_R/J_basis_from_tree}
Let $h=(h_1,\ldots,h_n)$ be a Hessenberg function.  Then
\begin{enumerate}
	\item The number of leaves in the $h$-tree at Level $n+1$ equals $\prod_{i=1}^n \beta_i$.
	\item The collection of leaf labels at Level $n+1$ in the $h$-tree is exactly the basis of monomials $\Bh$ of $R/J_h$ given by Theorem~\ref{thm:Basis_for_R/J}.
\end{enumerate}
\end{proposition}

\begin{example} Let $h=(2,3,3)$ be a Hessenberg function.  It has corresponding degree tuple $\beta=(2,2,1)$.  Figure~\ref{fig:h-tree_for_h=(2,3,3)} shows the corresponding $h$-tree.

\begin{figure}[!ht]
$$
\xymatrix{
{\mbox{Level 1}} & & & & {\bullet} \ar[dll]_{x_2} \ar[drr]^{1}\\
{\mbox{Level 2}} & & {\bullet} \ar[dl]_{x_3} \ar[dr]^{1} & & & & {\bullet} \ar[dl]_{x_3} \ar[dr]^{1}\\
{\mbox{Level 3}} & {\bullet} \ar[d]_{1} & & {\bullet} \ar[d]_{1} & & {\bullet} \ar[d]_{1} & & {\bullet} \ar[d]_{1}\\
{\mbox{Level 4}} & {x_2x_3} & & {x_2} & & {x_3} & & {1}
}
$$
%\vspace{-.25in}
\caption{\label{fig:h-tree_for_h=(2,3,3)}The $h$-tree for $h=(2,3,3)$.}
\end{figure}
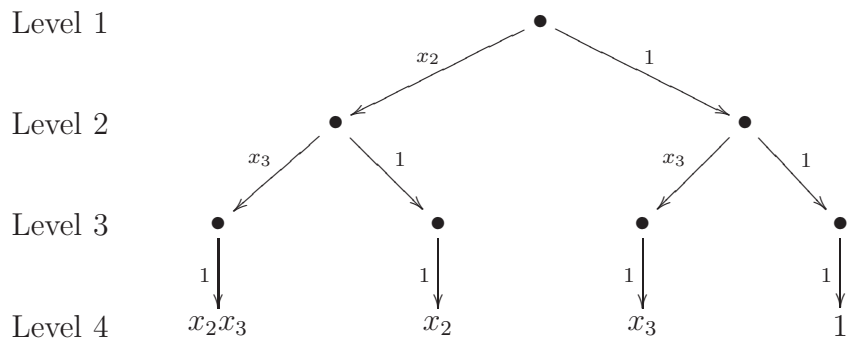
\end{example}

Before we give the precise construction of an $h$-tableau-tree, we define a barless tableau and give a lemma that instructs us how to fill this object to construct a tableau.

\begin{definition}[Barless tableau]
Fix $n$.  A \textit{barless tableau} is a diagram $\FiveTab$ filled with some proper subset of $\{1,\ldots,n\}$ without any bars.
\end{definition}

\begin{remark}[Using a barless tableau to build an $\JT$-filling]\label{rem:building_fillings}
We will place the values $1,2,\ldots, n$ into a barless tableau satisfying an $h$-permissibility condition.  When all $n$ numbers are in the barless tableau, we will introduce bars so that it is a traditional tableau.
\end{remark}

\begin{definition}[$h$-permissibility conditions]
Suppose we have placed the numbers $1,\ldots,i-1$ into a barless tableau.  We say that the numbers are in \textit{$h$-permissible} positions if each horizontal adjacency adheres to the rule: $k$ is immediately left of $j$ if and only if $k\leq h_j$.
\end{definition}

The lemma below allows us to predict how many $h$-permissible positions are available for the next value $i$.

\begin{lemma}\label{lem:bullet_correspondence}
Let $h=(h_1,\ldots,h_n)$ be a Hessenberg function.  If a barless tableau is filled with $1,2,\ldots,i-1$, then the number of $h$-permissible positions for $i$ in this tableau is exactly $\beta_i$, where $(\beta_n,\beta_{n-1},\ldots,\beta_1)$ is the degree tuple corresponding to $h$.
\end{lemma}

\begin{proof}
Let $h=(h_1,\ldots,h_n)$ be a Hessenberg function and $\beta=(\beta_n,\beta_{n-1},\ldots,\beta_1)$ be its corresponding degree tuple.  Suppose a barless tableau is filled with $1,2,\ldots,i-1$.  Consider $\beta_i$.  By definition $\beta_i = i - \#\{h_k \;|\; h_k < i\}$ and so $\#\{h_k \;|\;h_k<i\}$ equals $i-\beta_i$.  Since each $h_k$ is at least $k$, only the values $h_1,\ldots,h_{i-1}$ can possibly lie in the set $\{h_k \;|\; h_k<i\}$.  The remaining $(i-1) - (i-\beta_i) = \beta_i-1$ of the $h_1,\ldots,h_{i-1}$ satisfy $i\leq h_k$ which is the $h$-permissibility condition for the descent \setlength{\unitlength}{.15in}\begin{picture}(2,1)(0,0)
\linethickness{.25pt}
\multiput(0,0)(0,1){2}{\line(1,0){2}}
\multiput(0,0)(1,0){3}{\line(0,1){1}}
\put(.5,.5){\makebox(0,0){\begin{small}$i$\end{small}}}
\put(1.5,.5){\makebox(0,0){\begin{small}$k$\end{small}}}\end{picture}~.  Hence $i$ can be placed to the immediate left of any of these $\beta_i-1$ values.  This gives $\beta_i-1$ positions that are $h$-permissible positions.  In addition, the value $i$ can be placed to the right of the far-right entry since $i$ is larger than any number $1,\ldots,i-1$ in the barless tableau.  This yields a total of $(\beta_i-1)+1=\beta_i$ possible $h$-permissible positions for $i$.
\end{proof}

\begin{definition}[$h$-tableau-tree]\label{def:h_tab_tree}
Let $h=(h_1,\ldots,h_n)$ be a Hessenberg function and $\beta=(\beta_n,\beta_{n-1},\ldots,1)$ be its corresponding degree tuple.  The \textit{$h$-tableau-tree} is the $h$-tree together with an assignment of barless tableaux to label each vertex on Levels 1 to $n$.  The top is Level 1 and has a single barless tableau with the entry 1.  Given a barless tableau $T$ at Level $i-1$ with fillings $1,\ldots,i-1$, we obtain the $\beta_i$ different Level $i$ barless tableaux by the following algorithm:
\begin{itemize}
 \item Place a bullet at each of the $h$-permissible positions in the barless tableau $T$.  The diagram at Level $i$ joined by the edge $x_i^j$ is found by replacing the $(j+1)^{th}$ bullet (counting right to left) with the number $i$ and erasing all other bullets. 
\end{itemize}
When we reach Level $n$, each barless tableau will contain the numbers $1,\ldots,n$.  We may now place the bars into this tableau yielding a filling of $\mu$.
\end{definition}

\begin{remark}
Observe that travelling from a barless tableau at Level $i-1$ down to a barless tableau at Level $i$, Lemma~\ref{lem:bullet_correspondence} asserts there will be exactly $\beta_i$ bullets going right to left.  Hence $h$-tableau-trees are well-defined.
\end{remark}

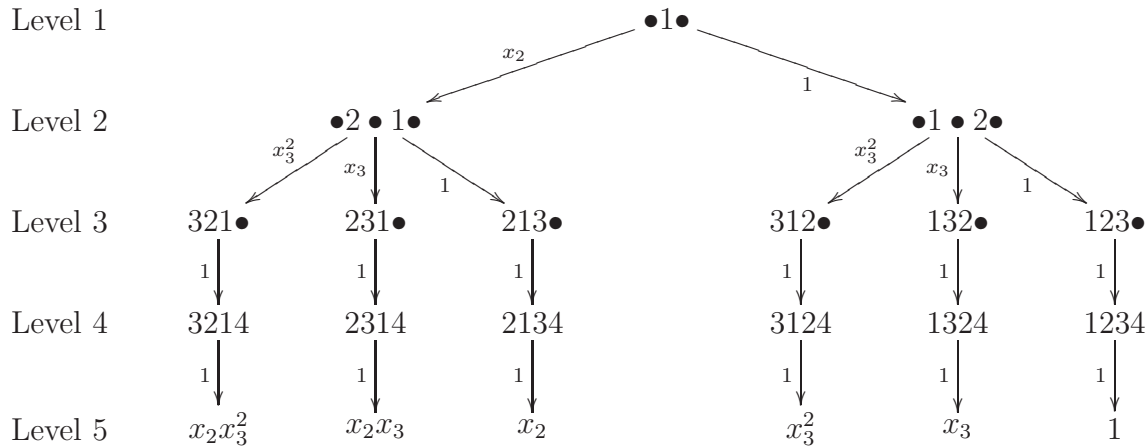
\begin{figure}[!ht]
\begin{center}
%\vspace{-.08in}
{\xymatrix{
{\mbox{Level 1}} & & & & {\bullet 1 \bullet} \ar[dll]_{x_2} \ar[drr]_1\\
{\mbox{Level 2}} & & {\bullet 2 \bullet 1 \bullet} \ar[dl]_{x_3^2} \ar[d]_{x_3} \ar[dr]_1 & & & & {\bullet 1 \bullet 2 \bullet} \ar[dl]_{x_3^2} \ar[d]_{x_3} \ar[dr]_1\\
{\mbox{Level 3}} & {321\bullet} \ar[d]_1 & {231\bullet} \ar[d]_1 & {213\bullet} \ar[d]_1 & & {312\bullet} \ar[d]_1 & {132\bullet} \ar[d]_1 & {123\bullet} \ar[d]_1\\
{\mbox{Level 4}} & {3214} \ar[d]_1 & {2314} \ar[d]_1 & {2134} \ar[d]_1 & & {3124} \ar[d]_1 & {1324} \ar[d]_1 & {1234} \ar[d]_1\\
{\mbox{Level 5}} & {x_2x_3^2} & {x_2x_3} & {x_2} & & {x_3^2} & {x_3} & {1}
}
}
%\vspace{-.1in}
\caption{\label{fig:tab_tree_example} The $h$-tableau-tree for $h=(3,3,3,4)$.}
\end{center}
\end{figure}
\vspace{-.25in}

%\medskip	%%%% REMOVE if figure is above here. and Replace the \vspace above

\begin{example}\label{exam:remarks_about_h-tab_tree_figure}
In Figure~\ref{fig:tab_tree_example}, we give the $h$-tableau-tree for $h=(3,3,3,4)$.  The corresponding degree tuple is $\beta=(1,3,2,1)$.  For ease of viewing, we omit the barless tableaux's rectangular boundaries and just give the fillings.  Observe that the six Level 4 tableaux are $\JT$-fillings.  There are only six possible $\JT$-fillings for this particular Hessenberg function and hence these are \textit{all} the $\JT$-fillings.  Further, the function $\Phi$ maps each one to the monomial in $\Bh$ on Level 5.  We conclude that the set of monomials $\mathcal{A}_{(3,3,3,4)}(\mu)$ coincides with the monomial basis $\mathcal{B}_{(3,3,3,4)}(\mu)$ for $R/J_h$ when using this regular nilpotent shape $\mu=(n)$.  We generalize these points in the next subsection, where we exhibit the inverse map to $\Phi$ in the setting of regular nilpotent Hessenberg varieties. Compare this with the elaborative example from Subsection~\ref{subsec:elab_example}.
\end{example}

\subsection{The inverse map \texorpdfstring{$\Psi_h$}{Psi-h} from monomials in \texorpdfstring{$\Bh$}{B-h(mu)} to \texorpdfstring{$\JT$}{(h,mu)}-fillings}\label{subsec:Psi_h_map}

Recall from Subsection~\ref{subsec:Phi_map} that the function $\Phi$ from $\JT$-fillings onto the set $\Ah$ of monomials is given by the map $$T \longmapsto \prod\limits_{\substack{(i,j) \in \D_j \\ 2 \leq j \leq n}} x_j.$$
In the Springer setting, we first constructed the inverse map $\Psi$ from $\A$ to $\JT$-fillings, then proved $\A=\B$.  In the regular nilpotent Hessenberg setting we will again prove that $\Phi$ is a graded vector space isomorphism, but this time we first construct an inverse map $\Psi_h$ from $\Bh$ and then verify $\Ah=\Bh$.  In this new setting this plan of attack is used since we know more about the structure of $\Bh$ (see Theorem~\ref{thm:Basis_for_R/J}), whereas in the Springer setting the basis $\B$ was given by Garsia and Procesi via a recursion formula~\cite[Equation 1.2]{GP}.  As the remarks in Example~\ref{exam:remarks_about_h-tab_tree_figure} disclosed, we will show the following:
\begin{enumerate}
 \item The Level $n$ fillings in the $h$-tableau-tree are distinct $\JT$-fillings.
 \item The number of $\JT$-fillings equals the number of leaves of the $h$-tableau-tree.
 \item The Level $n$ fillings are \textit{all} possible $\JT$-fillings.
 \item The function $\Phi$ maps each of these fillings to the monomial $\X \in \Bh$ below it at Level $n+1$.
 \item The set $\Ah$ coincides with the set $\Bh$.
\end{enumerate}

\begin{theorem}\label{thm:Level_n_fillings_are_JT}
Let $h=(h_1,\ldots,h_n)$ be a Hessenberg function.  The Level $n$ fillings of the corresponding $h$-tableau-tree are distinct $\JT$-fillings.
\end{theorem}
\begin{proof}
When going from Level $i-1$ down to $i$, the value $i$ is placed immediately to the left of a number $k\in\{1,\ldots,i-1\}$ only if $i\leq h(k)$.  That is, all fillings in the tree are $h$-permissible and hence the Level $n$ fillings are $\JT$-fillings.  Branching rules ensure all are distinct.
\end{proof}

The proof of Theorem~\ref{thm:number_of_JT-fillings_equals_product_of_betas} relies on combinatorial facts about the two numbers in question, namely the cardinalities of the set of possible $\JT$-fillings and the set of leaves of an $h$-tableau-tree.  The former number is given by the following theorem.

\begin{theorem}[Sommers-Tymoczko~\cite{ST}]\label{thm:Sommers_Tymockzo}
Let $h=(h_1,\ldots,h_n)$ be a Hessenberg function.  The number of $\JT$-fillings of a one-row diagram of shape $(n)$ equals $\prod_{i=1}^n \nu_i$ where $\nu_i=h_i-i+1$.
\end{theorem}

\begin{proof}[Proof remark]
The notation and terminology in the statement of this theorem differ much from the source~\cite{ST}.  Proof of this theorem arises from considering their Theorem~10.2 along with their definition of ideal exponents given in Definition~3.2.
\end{proof}

Fix a Hessenberg function $h=(h_1,\ldots,h_n)$.  Let $A_h$ denote the multiset $A_h:=\{\nu_i\}_{i=1}^n$.  Proposition~\ref{thm:leaves_remarks_and_R/J_basis_from_tree} shows that the number of leaves of the $h$-tree (and consequently of the $h$-tableau-tree) is $\prod_{i=1}^n \beta_i$ where each $\beta_i$ equals $i - \#\{h_k < i \}$.  Let $B_h$ denote the multiset $B_h:=\{\beta_i \}_{i=1}^n$.  We remind the reader that in a multiset order is ignored, but multiplicity matters.  For example, $\{1,2,3\}=\{2,1,3\}$ but $\{1,1,2\}\neq\{1,2\}$.  Before we prove Theorem~\ref{thm:number_of_JT-fillings_equals_product_of_betas}, we define a pictorial representation of a Hessenberg function that gives a visual manner in which to compute the degree tuple corresponding to a Hessenberg function.

\begin{definition}[Hessenberg diagram]\label{def:Hess_diagram}
Let $h=(h_1,\ldots,h_n)$ be a Hessenberg function.  We may represent $h$ pictorially by an $n$-by-$n$ grid of boxes where we shade the top $h_i$ boxes of column $i$, reading the columns left to right.  The constraints on $h$ force:
\begin{eqnarray*}
&(i)& i\leq h_i \; \Longrightarrow \mbox{ All shaded boxes in a column include the diagonal.}\\
&(ii)& h_i \leq h_{i+1} \; \Longrightarrow \mbox{ Every box to the right of any shaded box is also shaded.}
\end{eqnarray*}
Remove the strictly upper triangular subdiagram from this $h$-shading.  We call this the \textit{Hessenberg diagram} corresponding to $h$.
\end{definition}

\begin{example}\label{exam:h_334456}
Let $h=(3,3,4,4,5,6)$.  Then we have the following Hessenberg diagram:

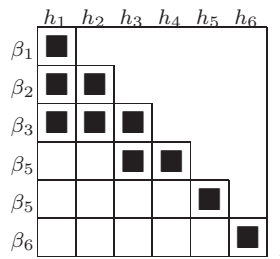
\begin{center}\setlength{\unitlength}{.2in}
\begin{picture}(6,6)(0,0)
\linethickness{.2pt}
\put(0,0){\line(1,0){6}}
\put(0,1){\line(1,0){6}}
\put(0,2){\line(1,0){5}}
\put(0,3){\line(1,0){4}}
\put(0,4){\line(1,0){3}}
\put(0,5){\line(1,0){2}}
\put(0,6){\line(1,0){1}}
\put(0,6){\line(1,0){6}}
\put(6,6){\line(0,-1){6}}
\put(0,0){\line(0,1){6}}
\put(1,6){\line(0,-1){6}}
\put(2,5){\line(0,-1){5}}
\put(3,4){\line(0,-1){4}}
\put(4,3){\line(0,-1){3}}
\put(5,2){\line(0,-1){2}}
\multiput(.5,5.5)(0,-1){3}{\makebox(0,0){$\blacksquare$}}
\multiput(1.5,4.5)(0,-1){2}{\makebox(0,0){$\blacksquare$}}
\multiput(2.5,3.5)(0,-1){2}{\makebox(0,0){$\blacksquare$}}
\multiput(3.5,2.5)(0,-1){1}{\makebox(0,0){$\blacksquare$}}
\multiput(4.5,1.5)(0,-1){1}{\makebox(0,0){$\blacksquare$}}
\multiput(5.5,.5)(0,-1){1}{\makebox(0,0){$\blacksquare$}}
\put(.15,6.1){\begin{scriptsize}$h_1$\end{scriptsize}}\put(1.15,6.1){\begin{scriptsize}$h_2$\end{scriptsize}}
\put(2.15,6.1){\begin{scriptsize}$h_3$\end{scriptsize}}\put(3.15,6.1){\begin{scriptsize}$h_4$\end{scriptsize}}
\put(4.15,6.1){\begin{scriptsize}$h_5$\end{scriptsize}}\put(5.15,6.1){\begin{scriptsize}$h_6$\end{scriptsize}}
\put(-.7,.25){\begin{scriptsize}$\beta_6$\end{scriptsize}}\put(-.7,1.25){\begin{scriptsize}$\beta_5$\end{scriptsize}}
\put(-.7,2.25){\begin{scriptsize}$\beta_5$\end{scriptsize}}\put(-.7,3.25){\begin{scriptsize}$\beta_3$\end{scriptsize}}
\put(-.7,4.25){\begin{scriptsize}$\beta_2$\end{scriptsize}}\put(-.7,5.25){\begin{scriptsize}$\beta_1$\end{scriptsize}}
\end{picture}~.
\end{center}
Columns are read from left to right, and rows are read from top to bottom.  Visually, we see the value of $h_i$ is $i-1$ plus the number of shaded boxes in column $i$.  Furthermore, $\beta_i$ is the number of shaded boxes in row $i$.  In fact, the number of shaded boxes in row $i$ equals $i$ minus the number of columns left of column $i$ whose shaded boxes do not reach the $i^{th}$ row---namely, the value $i - \#\{h_k | h_k < i\}$.  This is exactly the degree tuple entry $\beta_i$ as defined in Definition~\ref{def:beta_tuple}.  Thus the degree tuple corresponding to $h$ is $\beta = (1,1,2,3,2,1)$, reminding the reader that the tuple $\beta$ by convention is written as $(\beta_6,\beta_5,\ldots,\beta_1)$.
\end{example}

\begin{theorem}\label{thm:number_of_JT-fillings_equals_product_of_betas}
The number of $\JT$-fillings equals the number of leaves in the $h$-tableau-tree.
\end{theorem}
\begin{proof}
Let $h=(h_1,\ldots,h_n)$ be a Hessenberg function.  It suffices to show the multisets $A_h$ and $B_h$ are equal.  Represent the function $h$ pictorially by its corresponding Hessenberg diagram.  We may view the elements of $A_h$ as the following vector difference: $$(\nu_i)_{i=1}^n = (h_1,\ldots,h_n) - (0,1,\ldots,n-1).$$
So $\nu_i$ equals the number of shaded boxes on or below the diagonal in column $i$.  Regarding the multiset $B_h$, observe that each element $\beta_i$ is the number of shaded boxes on or left of the diagonal in row $i$ (as noted in the remark in Example~\ref{exam:h_334456}).  Thus it suffices to show each column length $\nu_i$ corresponds to exactly one row length $\beta_j$.  We induct on the Hessenberg function.

Consider the minimal Hessenberg function $h=(1,2,\ldots,n)$. This gives the following Hessenberg diagram:
\medskip
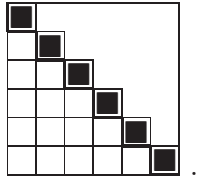
\begin{center}\setlength{\unitlength}{.15in}
\begin{picture}(6,6)(0,0)
\linethickness{.2pt}
\put(0,0){\line(1,0){6}}
\put(0,1){\line(1,0){6}}
\put(0,2){\line(1,0){5}}
\put(0,3){\line(1,0){4}}
\put(0,4){\line(1,0){3}}
\put(0,5){\line(1,0){2}}
\put(0,6){\line(1,0){1}}
\put(0,6){\line(1,0){6}}
\put(6,6){\line(0,-1){6}}
\put(0,0){\line(0,1){6}}
\put(1,6){\line(0,-1){6}}
\put(2,5){\line(0,-1){5}}
\put(3,4){\line(0,-1){4}}
\put(4,3){\line(0,-1){3}}
\put(5,2){\line(0,-1){2}}
\thicklines
\multiput(.5,5.5)(0,-1){1}{\makebox(0,0){$\blacksquare$}}
\multiput(1.5,4.5)(0,-1){1}{\makebox(0,0){$\blacksquare$}}
\multiput(2.5,3.5)(0,-1){1}{\makebox(0,0){$\blacksquare$}}
\multiput(3.5,2.5)(0,-1){1}{\makebox(0,0){$\blacksquare$}}
\multiput(4.5,1.5)(0,-1){1}{\makebox(0,0){$\blacksquare$}}
\multiput(5.5,.5)(0,-1){1}{\makebox(0,0){$\blacksquare$}}
\multiput(0,6)(1,-1){6}{\line(0,-1){1}}
\multiput(0,5)(1,-1){6}{\line(1,0){1}}
%\put(.5,.5){\makebox(0,0){4}}
\end{picture}~.
\end{center}
Each shaded box contributes to both a $\nu_i$ and a $\beta_i$ of length 1.  It follows that $A_h=B_h=\{1,1,\ldots,1\}$, proving the base case holds.

Assume for some fixed Hessenberg function $h=(h_1,\ldots,h_n)$ that $A_h=B_h$.  Add a shaded box to its Hessenberg diagram in a position $(i_0,j_0)$ so that the new function $\tilde{h}=(h_1,\ldots,h_{j_0-1},i_0,h_{j_0+1},\ldots,h_n)$ is a Hessenberg function, namely so $i_0 \leq h_{j_0+1}$.  We claim that the multisets $A_{\tilde{h}} = \{\tilde{\nu}_i\}_{i=1}^n$ and $B_{\tilde{h}} = \{\tilde{\beta_i} \}_{i=1}^n$ coincide.

Every box above $(i_0,j_0)$ in column $j_0$ must be shaded, up to the shaded diagonal box $(j_0,j_0)$.  This shaded column length is $\tilde{\nu}_{j_0}$.  And since $\tilde{h}$ is a Hessenberg function, every box to the right of $(i_0,j_0)$ is shaded up to the shaded diagonal box $(i_0,i_0)$.  This shaded row length is $\tilde{\beta}_{i_0}$.  No other box in row $i_0$ or column $j_0$ below the diagonal is shaded because $h$ is a Hessenberg function.  Clearly,
$$\tilde{\nu}_{j_0} = \nu_{j_0}+1 = (h_{j_0} - j_0 + 1)+1 = h_{j_0} - j_0 + 2.$$
The value $\tilde{\beta}_{i_0}$ is just the number of boxes in row $i_0$ from the position $(i_0,j_0)$ to the diagonal $(i_0,i_0)$ which we count is $i_0-j_0+1$.  Observe $i_0=h_{j_0}+1$ implies that $h_{j_0}+2 = i_0+1$.  Hence $h_{j_0}-j_0+2 = i_0-j_0+1$.  We conclude $\tilde{\nu}_{j_0} = \tilde{\beta}_{i_0}$, and the claim holds since

\begin{center}
\begin{tabular}{ll}
(1) & $\nu_{j_0} = \beta_{i_0}$ necessarily in the original Hessenberg diagram for $h$,\\
(2) & $\nu_{j_0}$ and $\beta_{i_0}$ both increase by 1 in the new Hessenberg diagram for $\tilde{h}$, and\\
(3) & no other $\nu_i$ or $\beta_j$ in the original diagram for $h$ will change in the diagram for $\tilde{h}$.
\end{tabular}
\end{center}
This completes the induction step, and we conclude that the multisets $A_h$ and $B_h$ are equal for all $h$.
\end{proof}

\begin{example}[Clarifying example for the induction step above]\label{exam:clarifying_h_diagram}
Let $h$ be the Hessenberg function $(3,3,4,4,5,6)$.  The corresponding Hessenberg diagram is
\medskip
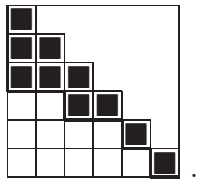
\begin{center}\setlength{\unitlength}{.15in}
\begin{picture}(6,6)(0,0)
\linethickness{.2pt}
\put(0,0){\line(1,0){6}}
\put(0,1){\line(1,0){6}}
\put(0,2){\line(1,0){5}}
\put(0,3){\line(1,0){4}}
\put(0,4){\line(1,0){3}}
\put(0,5){\line(1,0){2}}
\put(0,6){\line(1,0){1}}
\put(0,6){\line(1,0){6}}
\put(6,6){\line(0,-1){6}}
\put(0,0){\line(0,1){6}}
\put(1,6){\line(0,-1){6}}
\put(2,5){\line(0,-1){5}}
\put(3,4){\line(0,-1){4}}
\put(4,3){\line(0,-1){3}}
\put(5,2){\line(0,-1){2}}
\thicklines
\multiput(.5,5.5)(0,-1){3}{\makebox(0,0){$\blacksquare$}}
\multiput(1.5,4.5)(0,-1){2}{\makebox(0,0){$\blacksquare$}}
\multiput(2.5,3.5)(0,-1){2}{\makebox(0,0){$\blacksquare$}}
\multiput(3.5,2.5)(0,-1){1}{\makebox(0,0){$\blacksquare$}}
\multiput(4.5,1.5)(0,-1){1}{\makebox(0,0){$\blacksquare$}}
\multiput(5.5,.5)(0,-1){1}{\makebox(0,0){$\blacksquare$}}
\put(0,6){\line(0,-1){3}}
\put(0,3){\line(1,0){2}}
\put(2,3){\line(0,-1){1}}
\put(2,2){\line(1,0){2}}
\put(4,2){\line(0,-1){1}}
\put(4,1){\line(1,0){1}}
\put(5,1){\line(0,-1){1}}
\put(5,0){\line(1,0){1}}
%\put(.5,.5){\makebox(0,0){4}}
\end{picture}~.
\end{center}
In this example $A=\{3,2,2,1,1,1\}$ and $B=\{1,2,3,2,1,1\}$ reading the column lengths from left to right and row lengths from top to bottom, respectively.  At the induction step in the proof above, there are only three legal places to add a box: the positions $(4,2)$, $(5,4)$, or $(6,5)$.  Adding the $(4,2)$-box, for instance, changes $\nu_2$ from 2 to 3 and changes $\beta_4$ from 2 to 3 also.  Moreover, adding the $(4,2)$-box did not affect any other $\nu_i$ or $\beta_j$ values in $A_h$ or $B_h$ respectively.
\end{example}

\begin{corollary}\label{cor:all_possible_JT_fillings}
The Level $n$ fillings of the $h$-tableau tree are all possible $\JT$-fillings.
\end{corollary}
\begin{proof}
Level $n$ fillings are distinct $\JT$-fillings by Theorem~\ref{thm:Level_n_fillings_are_JT}.  The claim follows immediately from the previous theorem together with Theorem~\ref{thm:Sommers_Tymockzo} of Sommers-Tymoczko.
\end{proof}

We now introduce a lemma similar to Lemma~\ref{lem:why_this_box_exists} from the Springer setting.  This will be useful in building the inverse map $\Psi_h$.

\begin{lemma}\label{why_this_bullet_exists}
Fix $n$ and let $h$ be an arbitrary Hessenberg function.  Let $\X \in \Bh$, and consider the $h$-tableau-tree corresponding to $h$.  Then
\begin{enumerate}[(i)]%\setlength{\itemsep}{-5ex}
	\item The monomial $\X$ is of the form $x_2^{\alpha_2}\cdots x_n^{\alpha_n}$.  That is, no monomial in $\Bh$ contains the variable $x_1$.
	\item Every barless tableau at Level $i-1$ has at least $\alpha_i+1$ bullet positions available.
\end{enumerate}
\end{lemma}

\begin{proof}
Let $h=(h_1,\ldots,h_n)$ be a Hessenberg function and $\beta = (\beta_n,\beta_{n-1},\ldots,\beta_1)$ be its corresponding degree tuple .  Let $\X \in \Bh$.  By Theorem~\ref{thm:Basis_for_R/J}, $\X$ is of the form $x_1^{\alpha_1}x_2^{\alpha_2}\cdots x_n^{\alpha_n}$ where each $\alpha_i$ satisfies $0\leq\alpha_i\leq\beta_i-1$.  Since $\beta_1=1$ by definition, we have $\alpha_1=0$ for all $h$, proving ($i$).  Lemma~\ref{lem:bullet_correspondence} ensures that a Level $i-1$ barless tableau will have $\beta_i$ bullets.  Since $\alpha_i+1\leq\beta_i$, this proves ($ii$).
\end{proof}

\begin{theorem}[A map from $\Bh$ to $\JT$-fillings]\label{thm:Map_from_Bh_to_fillings}
Given a Hessenberg function $h$ and the shape $\mu=(n)$, there exists a well-defined dimension-preserving map $\Psi_h$ from the monomials $\Bh$ to the set of $\JT$-fillings.  That is, degree-$r$ monomials in $\Bh$ map to $r$-dimensional $\JT$-fillings.  Moreover the composition
$$\Bh \; \stackrel{\Psi_h}{\longrightarrow} \; \{\JT\mbox{-fillings}\} \; \stackrel{\Phi}{\longrightarrow} \; \Bh$$ is the identity on $\Bh$.
\end{theorem}
\begin{proof}
Let $\X\in\Bh$ have degree $r$.  Consider the $\JT$-filling $T$ sitting at Level $n$ directly above $\X$.  Define $\Psi_h(\X) := T$.  Lemma~\ref{why_this_bullet_exists} gives that $\X$ has the form $x_2^{\alpha_2}\cdots x_n^{\alpha_n}$.  Also since $\X$ has degree r, it follows that $\alpha_2+\alpha_3+\cdots+\alpha_n$ equals $r$.  It suffices to show that the cardinality of $\D_k$ equals $\alpha_k$ for each $k \in \{2,\ldots,n\}$.  We check this by examining the path on the $h$-tableau-tree from Level 1 down to $\X$ at Level $n+1$.  Fix $k\in\{2,\ldots,n\}$.  Let $T_1,T_2,\ldots,T_{n-1}$ be the barless tableaux on this path at Levels $1,2,\ldots,n-1$ respectively.  At the $(k-1)^{th}$ step in this path, the number $k$ is placed in the $(\alpha_k+1)^{th}$ bullet from the right in $T_{k-1}$.  This bullet exists by Lemma~\ref{why_this_bullet_exists}.  The barless tableau $T_{k-1}$ has the form
$$\setlength{\unitlength}{.15in}\begin{picture}(13,1.5)(0,0)
\linethickness{.25pt}
\multiput(0,0)(0,1.5){2}{\line(1,0){13}}
\put(0,0){\line(0,1){1.5}}
\put(13,0){\line(0,1){1.5}}
\put(.5,.4){$\cdots\bullet C_{\alpha_k}\bullet\cdots\bullet C_2\bullet C_1\bullet$}
\end{picture}$$
where each block $C_i$ is a string of numbers.  The numbers $1,\ldots,k-1$ are distributed without repetition amongst all $C_i$.  We claim there exists exactly one $c_i$ in each $C_i$-block to the right of $k$ in $T_k$ such that $(c_i,k)\in \D_k$.

Since all fillings in an $h$-tableau-tree are $h$-permissible, each block $C_i$ is an ordered string of $\gamma_i$ numbers $c_{i,1}c_{i,2}\cdots c_{i,\gamma_i}$ in $\{1,\ldots,k-1\}$ such that $c_{i,r} \leq h(c_{i,r+1})$ for each $r<\gamma_i$.  We claim that $k$ forms a dimension pair with only the far-right entry $c_{i,\gamma_i}$ of each block $C_i$ to its right.  That is, the value $k$ is not in a dimension pair with any other element of each block $C_i$ to its right.  Recall to be a dimension pair $(c,k) \in \D_k$ in the one-row case, we must have
\begin{enumerate}
 \item[($i$)] $c$ is to the right of $k$ and $k>c$ holds, and
 \item[($ii$)] if there exists a $j$ immediately right of $c$, then $k\leq h(j)$ holds also.
\end{enumerate}
Since $k$ is larger than every entry in the Level $k-1$ barless tableau, condition ($i$) holds.  If there exists no $j$ to the right of $c_{i,\gamma_i}$, then ($ii$) holds vacuously.  If some $j$ is eventually placed immediately right of $c_{i,\gamma_i}$ then $j\geq k+1$.  Thus $k<k+1\leq h(j)$ and so ($ii$) holds.  Lastly, if no $j$ is placed right of $c_{i,\gamma_i}$ and there exists a block $C_{i-1}$ immediately right of $C_i$ in the final tableau $T$, then the element $c_{i-1,1}$ is immediately right of $c_{i,\gamma_i}$.  But $k\leq h(c_{i-1,1})$ since $T_{k-1}$ had a bullet placed left of the block $C_{i-1}$.  Thus in every case, ($ii$) holds and $|\D_k|$ equals $\alpha_k$ as desired.

Hence the map $\Psi_h$ from $\Bh$ to the set of $\JT$-fillings takes degree-$r$ monomials to $r$-dimensional $\JT$-fillings, and  $\Phi\circ\Psi_h$ is the identity on $\Bh$.
\end{proof}

\begin{example}
Fix $h=(2,4,4,5,5)$ and its corresponding $\beta=(2,3,2,2,1)$.  The degree tuple $\beta$ tells us that the monomial $x_2x_4^2x_5$ lies in $\Bh$.  Without drawing the whole $h$-tableau tree, we can construct the unique path that gives the corresponding $\JT$-filling.  Omitting the barless tableau frames, we get
$$\bullet 1\bullet \;\buildrel x_2^1\over\longrightarrow\; \bullet2 1\bullet \;\buildrel x_3^0\over\longrightarrow\; \bullet2 1\bullet 3\bullet \;\buildrel x_4^2\over\longrightarrow\; \bullet4 2 1 3 \bullet \;\buildrel x_5^1\over\longrightarrow\; 54213.$$  Thus $\Psi_h(x_2x_4^2x_5)=T$ where $T$ is the $\JT$-filling \setlength{\unitlength}{.15in}\begin{picture}(5,1)(0,0)
\linethickness{.25pt}
\multiput(0,0)(0,1){2}{\line(1,0){5}}
\multiput(0,0)(1,0){6}{\line(0,1){1}}
\put(.25,.075){5}\put(1.25,.075){4}\put(2.25,.075){2}\put(3.25,.075){1}\put(4.25,.075){3}
\end{picture}~.  Conversely to recover the corresponding monomial from this filling $T$, we calculate the dimension-pairs.  We write all pairs $(ik)$ where $k$ is left of $i$ and $i<k$.  We then eliminate pairs $(ik)$ that do not satisfy the additional dimension pair condition that if $j$ is immediately right of $i$, then $k\leq h(j)$.  We get the following:
$$(12) \in \D_2, \;\;\; (14),\cancel{(24)},(34) \in \D_4, \;\;\; \textrm{and } \; \cancel{(15)},\cancel{(25)},(35),\cancel{(45)} \in \D_5.$$
Thus $\Phi$ takes the filling $T$ to the monomial $x_2x_4^2x_5$ as desired.

In particular, the algorithm for $\Psi_h$ depends on the choice of the Hessenberg function $h$.  If we considered the Hessenberg function $h'=(2,3,5,5,5)$, then the same filling $T= \setlength{\unitlength}{.15in}\begin{picture}(5,1)(0,0)
\linethickness{.25pt}
\multiput(0,0)(0,1){2}{\line(1,0){5}}
\multiput(0,0)(1,0){6}{\line(0,1){1}}
\put(.25,.075){5}\put(1.25,.075){4}\put(2.25,.075){2}\put(3.25,.075){1}\put(4.25,.075){3}
\end{picture}$ would be a permissible filling of $h'$, but now the dimension pair $(15)\in\D_5$ is \textit{not} cancelled since $5\leq h'(3)$.  Thus the map $\Phi$ takes $T$ to the monomial $x_2x_4^2x_5^2$.  Conversely, the inverse map $\Psi_{h'}$ now takes the new degree tuple into account and from this \textit{different} monomial we will get the same $T$ as we had gotten before.  The only thing that changes is the extra bullet before the last arrow:
$$\bullet 1\bullet \;\buildrel x_2^1\over\longrightarrow\; \bullet2 1\bullet \;\buildrel x_3^0\over\longrightarrow\; \bullet2 1\bullet 3\bullet \;\buildrel x_4^2\over\longrightarrow\; \bullet4 2 1 \bullet 3 \bullet \;\buildrel x_5^2\over\longrightarrow\; 54213.$$
\end{example}

\subsection{\texorpdfstring{$\Ah$}{A-h(mu)} coincides with the basis of monomials \texorpdfstring{$\Bh$}{B-h(mu)} for \texorpdfstring{$R/J_h$}{R/Jh} }\label{subsec:Ah_equals_Bh_RegNilpCase}
\begin{corollary}\label{cor:Ah_and_Bh_are_equal}
For a given $h$, the set of monomials $\Ah$ and $\Bh$ are equal.
\end{corollary}
\begin{proof}
The Level $n$ fillings are $\JT$-fillings by Theorem~\ref{thm:Level_n_fillings_are_JT}.  In fact they are all the possible $\JT$-filings by Corollary~\ref{cor:all_possible_JT_fillings}. Since the image of all $\JT$-fillings under $\Phi$ is $\Ah$, we have $\Ah=\Bh$.
\end{proof}

\begin{corollary}\label{cor:Ah_and_Mmu_isom_vect_spaces}
$\Ah$ and $M^{h,\mu}$ are isomorphic as graded vector spaces.
\end{corollary}

\begin{proof}
By Theorem~\ref{thm:Map_from_Bh_to_fillings}, the composition $\Phi\circ\Psi_h$ is the identity on $\Bh$.  Since $\Ah$ equals $\Bh$ and the number of paths in the $h$-tableau-tree is exactly $\prod_{i=1}^n \beta_i = |M^{h,\mu}|$, the cardinality of $\Ah$ equals the cardinality of the generating set of $\JT$-fillings in $M^{h,\mu}$.  Thus $\A$ and $M^{h,\mu}$ are isomorphic as graded vector spaces.
\end{proof}

We are now ready to state the theorem that ties the algebraic view of the $H^*(\mathfrak{H}(X,h))$ with the geometric view of this same cohomology ring.

\begin{theorem}\label{thm:algebraic_view_of_R/J_h}
Let $h=(h_1,\ldots,h_n)$ be a Hessenberg function with corresponding ideal $J_h$.  The generators of the quotient $R/J_h$ are in bijective graded correspondence with the $\JT$-fillings.  In particular, the generators of $R/J_h$ give the Betti numbers of the regular nilpotent Hessenberg varieties.
\end{theorem}

\begin{proof}
In Corollary~\ref{cor:Ah_and_Mmu_isom_vect_spaces}, we proved that the map $\Phi$ is a graded vector space isomorphism from $M^{h,\mu}$ to $\Ah$.  In particular it a bijective graded correspondence between the set of $\JT$-fillings and the set of monomials $\Ah$.  By Corollary~\ref{cor:Ah_and_Bh_are_equal}, the sets $\Ah$ and $\Bh$ coincide.  Hence $\Psi_h$ sends the set of generators of degree $i$ in $R/J_h$ to the $i$-dimensional $\JT$-fillings.  By Tymoczko~\cite[Theorem 1.1]{Tym}, the cardinality of the set of $i$-dimensional $\JT$-fillings equals the dimension of the degree-$2i$ part of $H^*(\mathfrak{H}(X,h))$.  Therefore, the degree-$i$ generators of $R/J_h$ give the $2i^{th}$ Betti number of $\mathfrak{H}(X,h)$.
\end{proof}

%%%%%%%%%%%%%%%%%%%%%%%%%%%%%%%%%%%%%%%%%%%%%%%%%%%%%%
%%%%%%%%%%%%%%%%%%%%%%%%%%%%%%%%%%%%%%%%%%%%%%%%%%%%%%
%%%%%%%%%%%%%%%%%%%%%%%%%%%%%%%%%%%%%%%%%%%%%%%%%%%%%%
%%%%%%%%%%%%%%%%%%%%%%%%%%%%%%%%%%%%%%%%%%%%%%%%%%%%%%
%%%%%%%%%%%%%%%%%%%%%%%%%%%%%%%%%%%%%%%%%%%%%%%%%%%%%%
%%%%%%%%%%%%%%%%%%%%%%%%%%%%%%%%%%%%%%%%%%%%%%%%%%%%%%
%%%%%%%%%%%%%%%%%%%%%%%%%%%%%%%%%%%%%%%%%%%%%%%%%%%%%%
%%%%%%%%%%%%%%%%%%%%%%%%%%%%%%%%%%%%%%%%%%%%%%%%%%%%%%
%%%%%%%%%%%%%%%%%% S E C T I O N  4 %%%%%%%%%%%%%%%%%%

\section{Tantalizing evidence, elaborative example, future work and questions}\label{sec:evidence}

\subsection{A conjecture and Peterson variety evidence}

\begin{conjecture}
Fix $\mu=(n)$ and let $h$ be a Hessenberg function.  The quotient $R/J_h$ is a presentation for the cohomology ring of the regular nilpotent Hessenberg variety $\mathfrak{H}(X,h)$.  Moreover, this gives the cohomology ring with integer coefficients.
\end{conjecture}

The family of regular nilpotent Hessenberg varieties contains a subclass of varieties called Peterson varieties.  These are the $\mathfrak{H}(X,h)$ for which $X$ is a regular nilpotent operator (equivalently, $\mu$ has shape $(n)$) and the Hessenberg function is defined as $h(i)=i+1$ for $i<n$ and $h(n)=n$.  Harada and Tymoczko~\cite{HT} recently gave the first general computation of this cohomology ring in terms of generators and relations.  Their presentation is given via a Monk-type formula.  Although computable, the presentation is computationally heavy.  Computer software such as Macaulay 2 is needed to produce small examples and exhibit a basis (via Gr\"obner basis reduction).  For small $n$, we explored the relationship between their presentation and mine.  Thus far, the two are isomorphic as rings.  Besides its ease of computation, a further advantage of my conjectural presentation is that it generalizes to all regular nilpotent Hessenberg varieties.

\subsection{An elaborative example}\label{subsec:elab_example}
\begin{example}\label{app:elab_example}
Let $h=(3,3,3,4)$ be our Hessenberg function.  The degree tuple is $\beta = (1,3,2,1)$, yielding the ideal $J_h = \langle \tilde{e}_1(x_4),\tilde{e}_3(x_3,x_4),\tilde{e}_2(x_2,x_3,x_4),\tilde{e}_1(x_1,x_2,x_3,x_4) \rangle$.  More formally, we write
$$J_h = \left(
\begin{array}{c}
x_4, \\
x_3^3 + x_3^2 x_4 + x_3 x_4^2 + x_4^3, \\
x_2^2 + x_2 x_3 + x_2 x_4 + x_3^2 + x_3 x_4 + x_4^2, \\
x_1 + x_2 + x_3 + x_4
\end{array}
\right).$$
Since the generators of $J_h$ form a Gr\"obner basis, the leading term ideal $\LTJ$ equals the monomial ideal $\langle x_4, x_3^3, x_2^2, x_1 \rangle$.  Basic commutative algebra results yield that the quotient $R/J_h$ has basis $\{\X \; | \; \X \notin \LTJ \}$ giving the set
$$\left\lbrace x_1^{\alpha_1} x_2^{\alpha_2} x_3^{\alpha_3} x_4^{\alpha_4} \; \vline \;
\begin{array}{c}
\alpha_1=0,\;\; \alpha_2=0,1\\
\alpha_3=0,1,2,\;\; \alpha_4=0
\end{array}
\right\rbrace.$$

So the basis $\Bh$ is $\{1,\; x_2,\; x_3,\; x_2 x_3,\; x_3^2,\; x_2 x_3^2 \}$ and agrees with Level 5 of the $h$-tableau-tree of Figure~\ref{fig:tab_tree_example} as expected.  It has the predicted size $\prod_{i=1}^4 \beta_i = 1\cdot2\cdot3\cdot1 = 6$.  And similarly we should have exactly six possible $\JT$-fillings.  These are given by the rule: a horizontal adjacency  \setlength{\unitlength}{.15in}\begin{picture}(2,1)(0,0)
\linethickness{.25pt}
\multiput(0,0)(0,1){2}{\line(1,0){2}}
\multiput(0,0)(1,0){3}{\line(0,1){1}}
\put(.5,.5){\makebox(0,0){$k$}}
\put(1.5,.5){\makebox(0,0){$j$}}
\end{picture}
is allowed only if $k\leq h(j)$.  Thus descents of the form \begin{picture}(2,1)(0,0)
\linethickness{.25pt}
\multiput(0,0)(0,1){2}{\line(1,0){2}}
\multiput(0,0)(1,0){3}{\line(0,1){1}}
\put(.5,.4){\makebox(0,0){3}}
\put(1.5,.4){\makebox(0,0){2}}
\end{picture}\;,
\begin{picture}(2,1)(0,0)
\linethickness{.25pt}
\multiput(0,0)(0,1){2}{\line(1,0){2}}
\multiput(0,0)(1,0){3}{\line(0,1){1}}
\put(.5,.4){\makebox(0,0){3}}
\put(1.5,.4){\makebox(0,0){1}}
\end{picture}\;, and 
\begin{picture}(2,1)(0,0)
\linethickness{.25pt}
\multiput(0,0)(0,1){2}{\line(1,0){2}}
\multiput(0,0)(1,0){3}{\line(0,1){1}}
\put(.5,.4){\makebox(0,0){2}}
\put(1.5,.4){\makebox(0,0){1}}
\end{picture} are allowed.  Notice the rule allows all adjacent ascents.  Below we give these fillings, their corresponding dimension pairs, and their corresponding monomials given by the function $\Phi$ on each filling.

\setlength{\unitlength}{.17in}
$$\begin{array}{cccccc}
\begin{picture}(4,1)(0,0)
\linethickness{.25pt}
\multiput(0,0)(0,1){2}{\line(1,0){4}}
\multiput(0,0)(1,0){5}{\line(0,1){1}}
\put(.5,.5){\makebox(0,0){1}}
\put(1.5,.5){\makebox(0,0){2}}
\put(2.5,.5){\makebox(0,0){3}}
\put(3.5,.5){\makebox(0,0){4}}
\end{picture} \;\; & \;\; \begin{picture}(4,1)(0,0)
\linethickness{.25pt}
\multiput(0,0)(0,1){2}{\line(1,0){4}}
\multiput(0,0)(1,0){5}{\line(0,1){1}}
\put(.5,.5){\makebox(0,0){2}}
\put(1.5,.5){\makebox(0,0){1}}
\put(2.5,.5){\makebox(0,0){3}}
\put(3.5,.5){\makebox(0,0){4}}
\end{picture} \;\; & \;\; \begin{picture}(4,1)(0,0)
\linethickness{.25pt}
\multiput(0,0)(0,1){2}{\line(1,0){4}}
\multiput(0,0)(1,0){5}{\line(0,1){1}}
\put(.5,.5){\makebox(0,0){1}}
\put(1.5,.5){\makebox(0,0){3}}
\put(2.5,.5){\makebox(0,0){2}}
\put(3.5,.5){\makebox(0,0){4}}
\end{picture} \;\; & \;\; \begin{picture}(4,1)(0,0)
\linethickness{.25pt}
\multiput(0,0)(0,1){2}{\line(1,0){4}}
\multiput(0,0)(1,0){5}{\line(0,1){1}}
\put(.5,.5){\makebox(0,0){2}}
\put(1.5,.5){\makebox(0,0){3}}
\put(2.5,.5){\makebox(0,0){1}}
\put(3.5,.5){\makebox(0,0){4}}
\end{picture} \;\; & \;\; \begin{picture}(4,1)(0,0)
\linethickness{.25pt}
\multiput(0,0)(0,1){2}{\line(1,0){4}}
\multiput(0,0)(1,0){5}{\line(0,1){1}}
\put(.5,.5){\makebox(0,0){3}}
\put(1.5,.5){\makebox(0,0){1}}
\put(2.5,.5){\makebox(0,0){2}}
\put(3.5,.5){\makebox(0,0){4}}
\end{picture} \;\; & \;\; \begin{picture}(4,1)(0,0)
\linethickness{.25pt}
\multiput(0,0)(0,1){2}{\line(1,0){4}}
\multiput(0,0)(1,0){5}{\line(0,1){1}}
\put(.5,.5){\makebox(0,0){3}}
\put(1.5,.5){\makebox(0,0){2}}
\put(2.5,.5){\makebox(0,0){1}}
\put(3.5,.5){\makebox(0,0){4}}
\end{picture}\\
\updownarrow & \updownarrow & \updownarrow & \updownarrow & \updownarrow & \updownarrow\\
\mbox{no pairs} & (12) & (23) & (12),(13) & (13),(23) & (12),(13),(23)\\
\updownarrow & \updownarrow & \updownarrow & \updownarrow & \updownarrow & \updownarrow\\
1 & x_2 & x_3 & x_2 x_3 & x_3^2 & x_2 x_3^2
\end{array}$$

\noindent To our delight, we get the same monomials from the $\JT$-fillings as the basis of $R/J_h$.
\end{example}

\subsection{Forthcoming work}\label{subsec:forthcoming_work}
In current joint work~\cite{Mb2} with Tymoczko, we provide a generalization $I_h$ of the Tanisaki ideal from the Springer setting into the new setting of regular nilpotent Hessenberg varieties.  We prove that these ideals $I_h$ coincide with the ideals $J_h$ which in turn give Gr\"obner bases presentations for a family of generalized Tanisaki ideals.  The equality of these two families of ideals provides further support that the quotient rings $R/J_h$ have some intimate connection with the cohomology rings of regular nilpotent Hessenberg varieties.

\subsection{Two open questions}

\noindent\textbf{Question 1:}\\
We showed in Subsection~\ref{subsec:Psi_map} that we have an inverse map $\Psi$ if we fix $h=(1,2,\ldots,n)$ and let $\mu$ vary.  We showed in Subsection~\ref{subsec:Psi_h_map} that we have an inverse map $\Psi_h$ if we fix $\mu=(n)$ and let $h$ vary.  Is there an inverse map $\Psi_{h,\mu}$ which incorporates both the $h$-function and the shape $\mu$?

\bigskip
\noindent\textbf{Question 2:}\\
Is there a direct topological proof that our quotient ring $R/J_h$ is the cohomology ring of the regular nilpotent Hessenberg varieties?

\bibliographystyle{plain}	% (uses file "plain.bst")
\bibliography{myrefs}

\end{document}